\documentclass{amsart}

\usepackage{color,amssymb}%,refcheck}
\usepackage[all]{xy}

\newcommand{\w}{\omega}
\newcommand{\IN}{\mathbb N}
\newcommand{\IZ}{\mathbb Z}

\newcommand{\F}{\mathcal F}

\newcommand{\Ra}{\Rightarrow}

\newtheorem{theorem}{Theorem}
\newtheorem{proposition}{Proposition}
\newtheorem{lemma}{Lemma}

\theoremstyle{definition}
\newtheorem{definition}[theorem]{Definition}

\title[Homeomorphisms of the space of non-zero integers with the Kirch topology]{Homeomorphisms of the space of non-zero integers\\ with the Kirch topology}
\author{Yaryna Stelmakh}
\address{Ivan Franko National University of Lviv, Ukraine}
\email{yarynziya@ukr.net}
\subjclass[2010]{Primary: 54D05, 54H99; Secondary: 11A41, 11N13.}

\begin{document}
\begin{abstract} The {\em Golomb} (resp. {\em Kirch}) topology on the set $\IZ^\bullet$ of nonzero integers is generated by the base consisting of arithmetic progressions $a+b\IZ=\{a+bn:n\in\IZ\}$ where $a\in\IZ^\bullet$ and $b$ is a (square-free) number, coprime with $a$. In 2019 Dario Spirito proved that the space of nonzero integers endowed with the Golomb topology admits only two self-homeomorphisms.  In this paper we prove an analogous fact for the space of nonzero integers endowed with the Kirch topology: it also admits exactly two self-homeomorphisms.
\end{abstract}

\maketitle

In this paper we describe the homeomorphism group of the space $\IZ^\bullet$ of nonzero integers endowed with the {\em Kirch topology} $\tau_K$, which is generated by the subbase consisting of the cosets $a+p\IZ$ where $a\in\IZ^\bullet$ and $p$ is a prime number that does not divide $a$. On the subspace $\IN$ of $\IZ^\bullet$ this topology was introduced by Kirch in \cite{Kirch}.

Banakh, Stelmakh and Turek \cite{BYT} prove that the subspace $\IN$ of $(\IZ^\bullet,\tau_K)$ is topologically rigid in the sense that each self-homeomorphism of $\IN$ endowed with the subspace topology $\tau_K{\restriction}\IN=\{U\cap\IN:U\in\tau\}$ is the identity map of $\IN$.

On the other hand, the space $(\IZ^\bullet,\tau_K)$ does admit a non-trivial self-homeomorpfism, namely the map $$j:\IZ^\bullet\to\IZ^\bullet,\quad j:x\mapsto -x.$$ In this paper we prove that this is the unique non-trivial self-homeomorphism of the topological space $(\IZ^\bullet,\tau_K)$. A similar result for the Golomb topology on $\IZ^\bullet$ was proved by Dario Spirito \cite{Spirito}. The topological rigidity of the Golomb topology on $\IN$ was proved by Banakh, Spirito and Turek in \cite{BST}.

\begin{theorem}\label{t:main} The space $\IZ^\bullet=\IZ\setminus\{0\}$ of nonzero integers endowed with the Kirch topology admits only two self-homeomorphisms.
\end{theorem}

The proof of this theorem follows the lines of the proof of the topological rigidity of the space $(\IN,\tau_K{\restriction}\IN)$ from \cite{BYT}. The proof is divided into 23  lemmas.
A crucial role in the proof belongs to the superconnectedness of the Kirch space and the superconnecting poset of the Kirch space, which is defined in Section~\ref{s:poset}.

%\begin{theorem} The set $\Pi$ of prime numbers is a dense metrizable subspace of the Golomb space $\IN_\tau$.
%\end{theorem}

%For any number $x\in\IN$ let $\Pi_x$ be the set of prime divisors of $x$.

%\begin{theorem}\label{mainH} Any homeomorphism $h:\IN_\tau\to\IN_\tau$ of the Golomb space has the following properties:
%\begin{enumerate}
%\item $h(1)=1$;
%\item $h(\Pi)=\Pi$;
%\item $\Pi_{h(x)}=h(\Pi_x)$ for every $x\in \IN$.
%\end{enumerate}
%\end{theorem}

%Theorem~\ref{mainH} implies that the Golomb space is not topologically homogeneous. This answer a problem \cite{Ban}, posed by the first author on Mathoverflow.

%\begin{problem} Is the Golomb space rigid?
%\end{problem}

%We recall that a topological space $X$ is {\em rigid} if its homeomorphism group is trivial.

\section{Four classical number-theoretic results}

%For two numbers $x,y$ by $\gcd(x,y)$ we denote their greatest common divisor, and by $x\dag y$ the greatest divisor of $x$, which is coprime with $y$.

By $\Pi$ we denote the set of prime numbers. For a number $x\in\IZ$ by $\Pi_x$ we denote the set of all prime divisors of $x$. Two numbers $x,y\in\IZ$ are {\em coprime} iff $\Pi_x\cap\Pi_y=\emptyset$.% (which is equivalent to saying that $\gcd(x,y)=1$).

%A number $q\in\IN$ is called {\em square-free} if it is not divided by the square $p^2$ of any prime number $p$.

%For a number $x\in\IN$ and a prime number $p$ let $l_p(x)$ be the largest integer number such that $p^{l_p(x)}$ divides $x$. The function $l_p(x)$ plays the role of logarithm with base $p$. A number $x$ is square-free if and only if $l_p(x)\le 1$ for any prime number $p$.

%A function $f:X\to Y$ is called {\em finite-to-one} if for each $y\in Y$ the preimage $f^{-1}(y)$ is finite.

% For a point $x\in\IN$ by $\tau_x=\{U\in\tau\colon x\in U\}$ we denote the family of open neighborhoods of $x$ in the Kirch topology $\tau_K$ on $\IN$.

In the proof of Theorem~\ref{t:main} we shall exploit the  following four known results of Number Theory. The first one is the famous Chinese Remainder Theorem (see. e.g. \cite[3.12]{J}).

\begin{theorem}[Chinese Remainder Theorem]\label{Chinese} If
$b_1,\dots,b_n\in\IZ$ are pairwise coprime numbers, then for any numbers $a_1,\dots,a_n\in\IZ$, the intersection $\bigcap_{i=1}^n(a_i+b_i\IN)$ is infinite.\end{theorem}

The second classical result is not elementary and is due to Dirichlet \cite[S.VI]{Dirichlet}, see also \cite[Ch.7]{Ap}.

\begin{theorem}[Dirichlet]\label{Dirichlet} For any coprime numbers $a,b\in\IN$ the  arithmetic progression $a+b\IN$  contains a prime number.
\end{theorem}

The third classical result is a recent theorem of Miha\u\i lescu \cite{Mih} who solved old Catalan's Conjecture \cite{Met}.

\begin{theorem}[Mih\u ailescu]\label{Mihailescu} If $a,b\in \{m^{n+1}:n,m\in \IN\}$, then $|a-b|=1$ if and only if $\{a,b\}=\{2^3,3^2\}$.
\end{theorem}

The fourth classical result we use is due to Karl Zsigmondy \cite{Zsigmondy}, see also \cite[Theorem 3]{Roitman}.

\begin{theorem}[Zsigmondy]\label{Zsigmondy} For integer numbers $a,n\in\IN\setminus\{1\}$ the inclusion $\Pi_{a^n-1}\subseteq\bigcup\limits_{0<k<n}\Pi_{a^k-1}$ holds if and only if one of the following conditions is satisfied:
\begin{enumerate}
\item $n=2$ and $a=2^k-1$ for some $k\in\IN$; then $a^2-1=(a+1)(a-1)=2^k(a-1)$;
\item $n=6$ and $a=2$; then $a^n-1=2^6-1=63=3^2\times 7=(a^2-1)^2\times(a^3-1)$.
\end{enumerate}
\end{theorem}

\section{Superconnected spaces and their superconnecting posets}\label{s:poset}

In this section we discuss superconnected topological spaces and some order structures related to such spaces.

First let us introduce some notation and recall some notions.

 For a set $A$ and $n\in\w$ let $[A]^n=\{E\subseteq A:|A|=n\}$ be the family of $n$-element subsets of $A$, and $[A]^{<\w}=\bigcup_{n\in\w}[A]^n$ be the family of all finite subsets of $A$. For a function $f:X\to Y$ and a subset $A\subseteq X$ by $f[A]$ we denote the image $\{f(a):a\in A\}$ of the set $A$ under the function $f$.

  For a subset $A$ of a topological space $(X,\tau)$ by $\overline{A}$ we denote the closure of $A$ in $X$. For a point $x\in X$ we denote by $\tau_x:=\{U\in\tau:x\in U\}$ the family of all open neighborhoods of $x$ in $(X,\tau)$.
A {\em poset} is an abbreviation for a partially ordered set.

A family $\F$ of subsets of a set $X$ is called a {\em filter} if
\begin{itemize}
\item $\emptyset\notin\F$;
\item for any $A,B\in\F$ their intersection $A\cap B\in\F$;
\item for any sets $F\subseteq E\subseteq X$ the inclusion $F\in\F$ implies $E\in\F$.
\end{itemize}

A topological space $(X,\tau)$ is called {\em superconnected} if for any $n\in\IN$ and non-empty open sets $U_1,\dots,U_n$ the intersection $\overline{U_1}\cap\dots\cap\overline{U_n}$ is non-empty. This allows us to define the filter
$$\F_\infty=\{B\subseteq X\colon \exists U_1,\dots,U_n\in\tau\setminus\{\emptyset\}\;\;(\overline{U_1}\cap\dots\cap\overline{U_n}\subseteq B)\},$$
called the {\em superconnecting filter} of $X$.

For every finite subset $E$ of $X$ consider the subfilter
$$\F_E:=\{B\subseteq X:\textstyle{\exists (U_x)_{x\in E}\in\prod_{x\in E}\tau_x\;\;(\;\bigcap_{x\in E}\overline{U_x}\subseteq B)}\}$$
of $\F_\infty$. Here we assume that $\F_\emptyset=\{X\}$.
It is clear that for any finite sets $E\subseteq F$ in $X$ we have $\F_E\subseteq \F_F$.

The family $$\mathfrak F=\{\F_E:E\in[X]^{<\w}\}\cup\{\F_\infty\}$$ is endowed with the inclusion partial order and is called the {\em superconnecting poset} of the superconnected space $X$. The filters $\F_\emptyset$ and $\F_\infty$ are the smallest and largest elements of the poset $\mathfrak F$, respectively.

 The following obvious lemma shows that the superconnecting poset $\mathfrak F$ is a topological invariant of the superconnected space.

\begin{proposition}\label{p:iso} For any homeomorphism $h$ of any superconnected topological space $X$, the map $$\tilde h:\mathfrak F\to\mathfrak F,\quad \tilde h\colon\F\mapsto \{h[A]\colon A\in\F\},$$ is an order isomorphism of the superconnecting poset $\mathfrak F$.
\end{proposition}

In the following sections we shall study the order properties of the poset $\mathfrak F$ for the Kirch space $(\IZ^\bullet,\tau_K)$ and shall exploit the obtained information in the proof of the topological rigidity of the Kirch space.

\section{Proof of Theorem~\ref{t:main}}

We divide the proof of Theorem~\ref{t:main} into 23 lemmas.%\footnote{SInce that moment all lemmas should be modified to the case of $\IZ^\bullet$ instead of $\IZ$.} Our first lemma  describe the closures of arithmetic progressions in the Kirch topology.

\begin{lemma}\label{basic} For any $a,b\in\IZ^\bullet$ the closure $\overline{a+b\IZ}$ of the arithmetic progression $a+b\IZ$ in the Kirch space $(\IZ^\bullet,\tau_K)$ is equal to
	%basic open set $a+b\IN_0$ in $\IN_\tau$ its closure
$$\IZ^\bullet\cap\bigcap_{p\in\Pi_b}\big(\{0,a\}+p\IZ\big).$$
\end{lemma}
% {\color{red}bla}
\begin{proof}
First we prove that $\overline{a+b\mathbb Z}\subseteq \{0,a\}+p\IZ$ for every $p\in\Pi_b$.
Take any point $x\in\overline{a+b\IZ}$ and assume that $x\notin p\IZ$. Then $x+p\IZ$ is a neighborhood of $x$ and hence the intersection $(x+p\IZ)\cap(a+b\IZ)$ is not empty. Then there exist $u,v\in\IZ$ such that $x+pu=a+bv$. Consequently, $x-a=bv-pu\in p\IZ$ and $x\in a+p\IZ$.

Next, take any point $x\in\IZ^\bullet\cap\bigcap_{p\in\Pi_b}(\{0,a\}+p\IZ)$. Given any  neighborhood $O_x$ of $x$ in $(\IZ^\bullet,\tau_K)$, we should prove that $O_x\cap(a+b\IZ)\ne\emptyset$. By the definition of the Kirch topology there exists a square-free number $d\in\IZ^\bullet$ such that $d,x$ are coprime and $x+d\IZ\subseteq O_x$.

If $\Pi_b\subseteq \Pi_x$, then $b,d$ are coprime and by Chinese Remainder Theorem $\emptyset \ne (x+d\IZ)\cap (a+b\IZ)\subseteq O_x\cap(a+b\IZ)$.
So, we can assume $\Pi_b\setminus\Pi_x\ne\emptyset$.
The choice of $x\in \bigcap_{p\in\Pi_b}(\{0,a\}+p\IZ)$ guarantees that $x\in \bigcap_{p\in\Pi_b\setminus \Pi_x}(a+p\IZ)=a+q\IZ$ where $q=\prod_{p\in \Pi_b\setminus\Pi_x}p$. Since the numbers $x$ and $d$ are coprime and $d$ is square-free, the greatest common divisor of $b$ and $d$ divides the number $q$. Since $x-a\in q\IZ$, the Euclides algorithm yields two numbers $u,v\in\IZ$ such that $x-a=bu-dv$, which implies that $O_x\cap (a+b\IZ)\supseteq (x+d\IZ)\cap(a+b\IZ)\ne\emptyset$.
\end{proof}

Lemma~\ref{basic} implies that the Kirch space $(\IZ^\bullet,\tau_K)$ is superconnected and hence possesses the superconnecting filter
$$\F_\infty=\big\{F\subseteq\IZ^\bullet:\exists U_1,\dots,U_n\in\tau_K\setminus\{\emptyset\}\quad\textstyle\big(\bigcap\limits_{i=1}^n\overline{U_i}\subseteq F\big)\big\}$$and the superconnecting poset
$$\mathfrak F=\{\F_E:E\in[\IZ^\bullet]^{<\w}\}\cup\{\F_\infty\}$$
consisting  of the filters $$\F_E=\big\{F\subseteq\IZ^\bullet:\textstyle{\exists (U_x)_{x\in E}\in\prod_{x\in E}\tau_x\;\;\big(\bigcap_{x\in E}\overline{U_x}\subseteq F}\big)\big\}.$$ Here for a point $x\in\IZ^\bullet$ by $\tau_x:=\{U\subseteq\IZ^\bullet:x\in U\}$ we denote the family of open neighborhoods of $x$ in the Kirch topology $\tau_K$.

For a nonempty finite subset $E\subseteq\IZ^\bullet$, let $\Pi_E=\bigcap_{x\in E}\Pi_x$ be the set of common prime divisors of numbers in the set $E$. Also let
$$A_E=\{p\in\Pi:\exists k\in \IN\;\;(E\subset \{0,k\}+p\IZ)\}.$$
Observe that $\Pi_E\subseteq A_E$ and $A_E\ne\emptyset$ because $2\in A_E$.
%{\color{red}Let $p\in\Pi_E$ and $x\in E$. Then $x=pl$ for some $l\in\IN$.}
If $E$ is a singleton, then $A_E=\Pi$;
%{\color{red}Let $E=\{x\}$ and $p\in\Pi$. There is a number $k\in\IN$ such
% that $p\mid (x-k)$   };
if $|E|\ge 2$, then $A_E\subseteq \Pi_x\cup\Pi_y\cup\Pi_{x-y}\subseteq\{2,\dots,\max E\}$ for any distinct numbers $x,y\in E$. This inclusion follows from

	\begin{lemma}\label{2ae} For any two-element set $E=\{x,y\}\subset \IZ^\bullet$ we have $A_E= \Pi_x\cup\Pi_y\cup\Pi_{x-y}$. %and
		%$$
		%\alpha_E(p)=\begin{cases}1&\mbox{if $p=2$};\\
		%(y\!\!\mod p)&\mbox{if $p\in (\Pi_x\cup\Pi_{x-y})\setminus\{2\}$};\\
		%(x\!\!\mod p)&\mbox{if $p\in (\Pi_y\cup\Pi_{x-y})\setminus\{2\}$}.
		%\end{cases}
		%$$
\end{lemma}
	
\begin{proof} %The number $p=2$ belongs to $A_E$ because $E\subset \{0,1\}+\IZ$.
Each number $p\in\Pi_x$ (resp. $p\in\Pi_y$) belongs to $A_E$ because $\{x,y\}\subset\{0,y\}+p\IZ$ (resp. $\{x,y\}\subset\{0,x\}+p\IZ\}$). Each number $p\in\Pi_{x-y}$ belongs to $A_E$ because $\{x,y\}\subset x+p\IZ\subseteq\{0,x\}+p\IZ$. This proves that $\Pi_x\cup\Pi_y\cup\Pi_{x-y}\subseteq A_E$.

Now take any prime number $p\in A_E$ and assume that $p\notin \Pi_x\cup\Pi_y$. It follows from $\{x,y\}=E\subset\{0,\alpha_E(p)\}+p\IZ$ that $\{x,y\}\subseteq \alpha_E(p)+p\IZ$ and hence $x-y\in p\IZ$ and $p\in\Pi_{x-y}$.
\end{proof}

% Indeed, assuming that $A_E$ contains some prime number $p>\max E$, we can find a number $k\in \{1,\dots,p-1\}$ such that $E\subseteq \{0,k\}+p\IZ$. Then for any distinct numbers $x,y\in E$ we get $x,y\in k+p\IZ$ and hence $x-y\in p\IZ$ which is not possible as $p>\max E\ge|x-y|$.

Let $\alpha_E\colon A_E\to\w$ be the unique function satisfying the following conditions:
\begin{itemize}
\item[\textup{(i)}] $\alpha_E(p)<p$ for all $p\in A_E$;
\item[\textup{(ii)}] $E\subseteq\{0,\alpha_E(p)\}+p\IZ$ for all $p\in A_E$;
\item[\textup{(iii)}] $\alpha_E(2)=1$ and $\alpha_E(p)=0$ for all $p\in \Pi_E\setminus\{2\}$.
\end{itemize}

%For two numbers $n\in\IN$ and $z\in\IZ$ by $(z\!\mod n)$ we denote the unique number $x\in\{0,\dots,n-1\}$ such that $z-x\in n\IZ$.

%{\color{red}

\begin{lemma}\label{l:realization} Let $A\subset\Pi$ be a finite set containing $2$ and $\alpha:A\to\IN_0$ be a function such that $\alpha(2)=1$ and $\alpha(p)\in\{0,\dots,p-1\}$ for all $p\in A\setminus\{2\}$. Let  $x$ be the product of odd prime numbers in the set $A$ and $y$ be any number in the set $\IZ^\bullet\cap\bigcap_{p\in A}(\alpha(p)+p\IZ)$. Then the set $E=\{y,x,2x\}$ has $A_E=A$ and $\alpha_E=\alpha$.
\end{lemma}

\begin{proof} For every prime number $p\in A$ we have $\{x,y\}\subset\{0,y\}+p\IZ$, which implies that $p\in A_E$. Assuming that $A_E\setminus A$ contains some prime number $p$, we conclude that $x\notin p\IZ$ and hence the inclusion $\{y,x,2x\}=E\subset \{0,\alpha_E(p)\}+p\IZ$ implies $\{x,2x\}\subset\alpha_E(p)+p\IZ$ and $x=2x-x\in p\IZ$. This contradiction shows that $A_E=A$. To show that $\alpha_E=\alpha$, take any prime number $p\in A=A_E$.
If $p=2$, then $\alpha(p)=1=\alpha_E(p)$. So, we assume that $p\ne 2$. If $\alpha(p)=0$, then $y\in \alpha(p)+p\IZ=p\IZ$ and hence $p\in \Pi_E$. In this case $\alpha_E(p)=0=\alpha(p)$. If $\alpha(p)\ne 0$, then the number $y\in\alpha(p)+p\IZ$ is not divisible by $p$ and then the inclusions $\{y,x,2x\}\subseteq\{0,\alpha(p)\}+p\IZ$ and $\{y,x,2x\}=E\subset\{0,\alpha_E(p)\}+p\IZ$ imply that $\alpha(p)=\alpha_E(p)$.
\end{proof}

The following lemma yields an arithmetic description of the filters $\F_E$.

% aLet $p\in A_E\setminus\{2\}$. Assume that $p\notin\Pi_x$ and observe that $\{x,y\}=E\subseteq\{0,\alpha_E(p)\}+p\IZ$ implies $x\in \alpha_E(p)+p\IZ$ and hence $\alpha_E(p)=x\!\mod p$. If $y\in p\IZ$, then $p\in\Pi_y$. Otherwise $y\in k+p\IZ$ and then $p\in\Pi_{x-y}$.
			%Conversely, if $p\in\Pi_x\cup\Pi_{x-y}$   	
	%then $E\subseteq \{0,y\}+p\IZ$ and similarly if $p\in\Pi_y$ then $E\subseteq \{0,x\}+p\IZ$.
	%In both of the above cases $p\in A_E$.
%\end{proof}%}

%\begin{example}\label{ex:1x} For any number $x>1$ and the set $E=\{1,x\}$ we have
%$$\Pi_E=\emptyset,\;\;A_E=\Pi_x\cup \Pi_{x-1}\;\mbox{ and }\alpha_E(A_E)=\{1\}.$$
%\end{example}

%\begin{example}\label{ex:2o} For any odd number $x>2$ and the set $E=\{2,x\}$ we have
%$$\Pi_E=\emptyset,\;\;A_E=\{2\}\cup \Pi_x\cup \Pi_{x-2}\;\mbox{ \ and \ }\alpha_E(2)=1, \;\alpha_E(A_E\setminus\{2\})=\{2\}.$$
%\end{example}

%\begin{example}\label{ex:2e} For any even number $x>2$ and the set $E=\{2,x\}$ we have
%$$\Pi_E=\{2\},\;\;A_E=\Pi_x\cup \Pi_{x-2}\setminus \{2\}\mbox{ \ and \ }\alpha_E(A_E)\subseteq \{2\}.$$
%\end{example}

\begin{lemma}\label{l:Kirch} For any finite subset $E\subseteq\IZ^\bullet$ with $|E|\ge 2$ we have
$$\F_E=\big\{B\subseteq\IZ^\bullet\colon \exists L\in[\Pi\setminus A_E]^{<\w}\quad\bigcap_{p\in L}p\IZ^\bullet\cap\bigcap_{p\in A_E\setminus\Pi_E}(\{0,\alpha_E(p)\}+p\IZ)\subseteq B\big\}.$$Here we assume that $\bigcap_{p\in\emptyset}p\IZ^\bullet=\IZ^\bullet$.
\end{lemma}

\begin{proof} It suffices to verify  two properties:
\begin{enumerate}
\item for any $(U_x)_{x\in E}\in\prod_{x\in E}\tau_x$ there exists a finite set $L\subseteq \Pi\setminus A_E$ such that
$$\bigcap_{p\in L}p\IZ^\bullet\cap\bigcap_{p\in A_E\setminus\Pi_E}(\{0,\alpha_E(p)\}+p\IZ)\subseteq \bigcap_{x\in E}\overline{U_x};$$
\item for any finite set $L\subseteq\Pi\setminus A_E$ there exists a sequence of neighborhoods $(U_x)_{x\in E}\in\prod_{x\in E}\tau_x$ such that
$$\bigcap_{x\in E}\overline{U_x}\subseteq \bigcap_{p\in L}p\IZ^\bullet\cap\bigcap_{p\in A_E\setminus\Pi_E}(\{0,\alpha_E(p)\}+p\IZ).$$
\end{enumerate}

1. Given a sequence of neighborhoods $(U_x)_{x\in E}\in\prod_{x\in E}\tau_x$, for every $x\in E$ find a square-free number $q_x>x$ such that $\Pi_{q_x}\cap \Pi_x=\emptyset$ and $x+q_x\IZ\subseteq U_x$. We claim that the finite set $L=\bigcup_{x\in E}\Pi_{q_x}\setminus A_E$ has the required property. Given any number $z\in \bigcap_{p\in L}p\IZ^\bullet\cap\bigcap_{p\in A_E\setminus\Pi_E}(\{0,\alpha_E(p)\}+p\IZ)$, we should prove that $z\in\overline{U_x}$ for every $x\in E$. By  Lemma~\ref{basic},
$$\IZ^\bullet\cap \bigcap_{p\in\Pi_{q_x}}(\{0,x\}+p\IZ)=\overline{ (x+q_x\IZ)}\subseteq\overline{U_x}.$$ So, it suffices to show that $z\in \{0,x\}+p\IZ$ for any $p\in \Pi_{q_x}$. Since the numbers $x$ and $q_x$ are coprime, $p\notin\Pi_x$ and hence $p\notin\Pi_E$. If $p\notin A_E$, then $p\in \Pi_{q_x}\setminus  A_E\subseteq L$ and hence $z\in p\IN\subseteq \{0,x\}+p\IZ$. If $p\in A_E$, then $x\in E\subseteq\{0,\alpha_E(p)\}+p\IZ$ and $x\in\alpha_E(p)+p\IZ$ (as $p\notin\Pi_x$). Then $x+p\IZ=\alpha_E(p)+p\IZ$ and
$z\in \{0,\alpha_E(p)\}+p\IZ=\{0,x\}+p\IZ.$

2. Fix any finite set $L\subseteq \Pi\setminus  A_E$. For every $x\in E$ consider the neighborhood $U_x=\bigcap_{p\in L\cup A_E\setminus\Pi_x}(x+p\IZ)$ of $x$ in the Kirch topology. By Lemma~\ref{basic},
$$\overline{U_x}=\IZ^\bullet\cap\bigcap_{p\in L\cup A_E\setminus \Pi_x}(\{0,x\}+p\IZ).$$

Given any number $z\in\bigcap_{x\in E} \overline{U_x}$, we should show that $z\in  \bigcap_{p\in L}p\IZ^\bullet\cap\bigcap_{p\in A_E\setminus \Pi_E}(\{0,\alpha_E(p)\}+p\IZ)$.
This will follow as soon as we check that $z\in p\IZ^\bullet$ for all $p\in L$ and $z\in\{0,\alpha_E(p)\}+p\IZ$ for all $p\in A_E\setminus \Pi_E$.

Given any $p\in A_E\setminus\Pi_E$, we can find a point $x\in E\setminus p\IZ$ and observe that $x\in E\subseteq\{0,\alpha_E(p)\}+p\IZ$. Then $z\in \overline{U_x}\subseteq\overline{x+p\IZ}\subseteq \{0,x\}+p\IZ=\{0,\alpha_E(p)\}+p\IZ$.

Now take any prime number $p\in L$. Since $L\cap  A_E=\emptyset$, we conclude that $E\not\subseteq p\IZ$. So, we can fix a number $x\in E\setminus p\IZ$. Taking into account that $p\notin A_E$, we conclude that $E\not\subseteq \{0,x\}+p\IZ$ and hence there exists a number $y\in E$ such that $p\IZ\ne y+p\IZ\ne x+p\IZ$. Then
$$z\in\overline{U_x}\cap\overline{U_y}\subseteq(\{0,x\}+p\IZ)\cap(\{0,y\}+p\IZ)=p\IZ.$$
\end{proof}

We shall use Lemma~\ref{l:Kirch} for an arithmetic characterization of the partial order of the superconnecting poset $\mathfrak F$ of the Kirch space.

\begin{lemma}\label{l:wo2} For two finite subsets $E,F\subseteq \Pi$ with $\min\{|E|,|F|\}\ge2$ we have $\F_E\subseteq \F_F$ if and only if $$A_F\subseteq A_E,\;\; \Pi_F\setminus\{2\}\subseteq \Pi_E \mbox{ \  and \ }\alpha_E{\restriction}A_F\setminus\Pi_E=\alpha_F{\restriction}A_F\setminus\Pi_E.$$
\end{lemma}

\begin{proof} To prove the ``only if'' part, assume that $\F_E\subseteq\F_F$.  By Lemma~\ref{l:Kirch}, the set $$\bigcap_{p\in A_F\setminus A_E}p\IZ^\bullet\cap\bigcap_{p\in A_E\setminus\Pi_E}(\{0,\alpha_E(p)\}+p\IZ)$$ belongs to the filter $\F_E\subseteq\F_F$.
By Lemma~\ref{l:Kirch}, there exists a finite set $L\subset\Pi\setminus A_F$ such that
\begin{equation}\label{eq1}
 \bigcap_{p\in L}p\IZ^\bullet\cap\bigcap_{p\in A_F\setminus\Pi_F}(\{0,\alpha_F(p)\}+p\IZ)\subseteq \bigcap_{p\in A_F\setminus A_E}p\IZ^\bullet\cap \bigcap_{p\in A_E\setminus\Pi_E}(\{0,\alpha_E(p)\}+p\IZ).
\end{equation}
This inclusion combined with the Chinese Remainder Theorem~\ref{Chinese} implies $$A_F\setminus A_E\subseteq L\subset \Pi\setminus A_F,\;\;A_E\setminus (\Pi_E\cup\{2\})\subseteq L\cup (A_F\setminus\Pi_F)\mbox{ \  and $\alpha_E(p)=\alpha_F(p)$ for any $p\in (A_F\setminus\Pi_F)\cap (A_E\setminus\Pi_E)$,}$$
and
\begin{equation}\label{eq2}
A_F\subseteq A_E,\;\;\Pi_F\setminus\{2\}\subseteq\Pi_E\mbox{ \  and \ }\alpha_E{\restriction}A_F\setminus\Pi_E=\alpha_F{\restriction}A_F\setminus\Pi_E.
\end{equation}

To prove the ``if'' part, assume that the condition (\ref{eq2}) holds. To prove that $\F_E\subseteq\F_F$, fix any set $\Omega\in\F_E$ and using Lemma~\ref{l:Kirch}, find a finite set $L\subseteq \Pi\setminus A_E$ such that
$$\bigcap_{p\in L}p\IZ^\bullet\cap\bigcap_{p\in A_E\setminus\Pi_E}(\{0,\alpha_E(p)\}+p\IZ)\subseteq\Omega.$$
Consider the finite set $\Lambda=(L\cup A_E)\setminus A_F=L\cup (A_E\setminus A_F)\supseteq L$ and observe that the condition (\ref{eq2}) implies the inclusion
\begin{equation}\label{eq3}\F_F\ni
\bigcap_{p\in\Lambda}p\IZ^\bullet\cap\bigcap_{p\in A_F\setminus\Pi_F}(\{0,\alpha_F(p)\}+p\IZ)\subseteq   \bigcap_{p\in L}p\IZ^\bullet\cap\bigcap_{p\in A_E\setminus \Pi_E}(\{0,\alpha_E(p)\}+p\IZ)\subseteq\Omega,
\end{equation}
yielding $\Omega\in\F_F$.
\end{proof}

\begin{lemma}\label{l:wo1} For two nonempty subsets $E,F\subseteq \IN$ with $\min\{|E|,|F|\}=1$ the relation $\F_E\subseteq\F_F$ holds if and only if $|E|=1$ and $E\subseteq F$.
\end{lemma}

\begin{proof} The ``if'' part is trivial. To prove the ``only if'' part, assume that $\F_E\subseteq\F_F$. First we prove that $|E|=1$. Assuming that $|E|>1$ and taking into account that $\min\{|E|,|F|\}=1$, we conclude that $|F|=1$. Choose a prime number $p>\max(E\cup F)$. Since $\bigcap_{y\in E}\overline{y+p\IZ}\in \F_E\subseteq\F_F$, for the unique number $x$ in the set $F$, there exists a square-free number $d$ such that $\Pi_d\cap\Pi_x=\emptyset$ and $\overline{x+dp\IZ}\subseteq\bigcap_{y\in E}\overline{y+p\IZ}$. By Lemma~\ref{basic},
$$x+qp\IZ^\bullet\subseteq \overline{x+dp\IZ}\subseteq\bigcap_{y\in E}\overline{y+p\IZ}=\bigcap_{y\in E}(\{0,y\}+p\IZ)=p\IZ.$$
The latter equality follows from $p>\max E$ and $|E|>1$. Then $x+dp\IZ^\bullet\subseteq p\IZ$ implies $x\in p\IZ$, which contradicts the choice of $p>\max (E\cup F)\ge x$. This contradiction shows that $|E|=1$. Let $z$ be the unique element of the set $E$.

It remains to prove that $z\in F$. To derive a contradiction, assume that $z\notin F$. Take any odd prime number $p>\max(E\cup F)$ and consider the set $\{0,z\}+p\IZ=\overline{z+p\IZ}\in\F_E\subseteq\F_F$. By the definition of the filter $\F_F$, for every $x\in F$ there exists a square-free number $d_x$ such that $\Pi_{d_x}\cap\Pi_x=\emptyset$ and $$\bigcap_{x\in F}\overline{x+d_x\IZ}\subseteq \overline{z+p\IZ}=\{0,z\}+p\IZ.$$ Consider the set $P=\bigcup_{x\in F}\Pi_{d_x}$.
 If $p\in \Pi_{d_x}$ for some $x\in F$, we can use the Chinese Remainder Theorem~\ref{Chinese} and find a number $$c\in (x+p\IZ)\cap \bigcap_{q\in P\setminus\{p\}}q\IZ\subseteq \bigcap_{y\in F}\overline{y+d_y\IZ}\subseteq \{0,z\}+p\IZ.$$
 Taking into account that $x$ is not divisible by $p$, we conclude that $c\in (x+p\IZ)\cap(z+p\IZ)$ and hence $x-z\in p\IZ$, which contradicts the choice of $p>\max(E\cup F)$. This contradiction shows that $p\notin P$. Since $p\ge 3$, we can find a number $z'\notin\{0,z\}+p\IZ$ and using the Chinese Remainder Theorem~\ref{Chinese}, find a number
 $$u\in (z'+p\IZ)\cap\bigcap_{q\in P}q\IZ^\bullet\subseteq \bigcap_{y\in F}\overline{y+d_y\IZ}\subseteq\{0,z\}+p\IZ,$$ which is a desired contradiction showing that $E\subseteq F$.
\end{proof}

%Lemma~\ref{l:wo2} implies the following important fact.

%\begin{lemma}\label{l:wo} For any $\F\in \mathfrak F$ with $|E|\ge 2$ the upper set ${\uparrow}\F=\{\mathcal E\in\mathfrak F:\F\subseteq \mathcal E\}$ is finite.
%\end{lemma}

%The following problem seems to be of crucial importance.

%\begin{problem} Fiven a finite set $E\subset\IN$ describe the order structure of the upper set ${\uparrow}\F_E$ in the poset $\mathfrak F$.
%\end{problem}

As we know, the largest element of the superconnecting poset $\mathfrak F$ is the superconnecting filter $\F_\infty$. This filter can be characterized as follows.

\begin{lemma}\label{filter0} The superconnecting filter $\F_\infty$ of the Kirch space  is generated by the base consisting of the sets $q\IN$ for an odd square-free number $q\in\IN$, i.e.
$$\F_\infty=\{B\subseteq \IZ^\bullet\colon \exists q\in (2\IN-1)\setminus\bigcup_{p\in\Pi} p^2\IN\;\;(q\IZ^\bullet\subseteq B)\}.$$
\end{lemma}

\begin{proof} Lemma~\ref{basic} implies that each element $F\in\F_\infty$ contains the set $q\IZ^\bullet$ for some odd square-free number $q$. Conversely, let $q$ be an odd square-free number. Then the sets $U_1=1+q\IZ$ and $U_2=2+q\IZ$ are open in the Kirch topology on $\IZ^\bullet$. By Lemma~\ref{basic} we have
		$$\overline{U_1}\cap\overline{U_2}=\mathbb{Z}^\bullet\cap\bigcap_{p\in\Pi_q}(\{0,1\}+p\IZ)\cap(\{0,2\}+p\IZ)=\IZ^\bullet\cap\bigcap_{p\in\Pi_q}p\IZ=  q\IZ^\bullet.$$
		Hence $q\IZ^\bullet\in\F_\infty$.
\end{proof}

\begin{lemma}\label{l:2max} For a nonempty subset $E\subseteq \IZ^\bullet$ the following conditions are equivalent:
\begin{enumerate}
\item $\F_E=\F_\infty$;
\item $A_E=\{2\}$.
\end{enumerate}
If  $|E|=2$, then the conditions \textup{(1), (2)} are equivalent to
\begin{enumerate}
\item[(3)]  $E\in\big\{\{2^n,2^{n+1}\},\{-2^n,-2^{n+1}\},\{-2^n,2^n\}:n\in\w\big\}$.
\end{enumerate}
\end{lemma}

\begin{proof} %$(2)\Ra(1)$ Assume that $E=\{2^n,2^{n+1}\}$ for some $n\in\w$. It follows that $A_E=\{2\}$ and $\alpha_E(A_E)=\{1\}$. Then $\{0,\alpha_E(2)\}+2\IZ=\IZ$ and by Lemmas~\ref{l:Kirch} and \ref{filter0},
%$$\F_D=\{F\subseteq \IN:\exists L\in[\Pi\setminus\{2\}]^{<\w}\;\;\bigcap_{p\in L}p\IN\subseteq F\}=\F_\infty.$$
%\smallskip

%$(1)\Ra(2)$
%{\color{red}
$(1)\Rightarrow(2)$: Assume $\F_{E}=\F_\infty$. Consider the set $F=\{1,2\}$ and observe that $A_F=\Pi_1\cup\Pi_2\cup\Pi_{2-1}=\{2\}$ and $\Pi_F=\emptyset$. Thus $A_F\subseteq A_E$, $\Pi_F{\setminus}\{2\}\subseteq \Pi_E$ and $\alpha_F{\restriction}A_F{\setminus} \Pi_E=\alpha_E{\restriction} A_F{\setminus} \Pi_E$. Lemma~\ref{l:wo2} implies $\F_{E}\subseteq \F_F$. Since $\F_{E}=\F_\infty$ is the largest element of $\mathfrak{F}$ we get $\F_E=\F_F$. By using again Lemma~\ref{l:wo2} we get $A_E\subseteq A_F$ which implies that $A_E = \{2\}$.
\vskip3pt	
	
$(2)\Rightarrow(1)$:
%\footnote{\color{blue} What about the implication $(1)\Rightarrow(2)$?
%At the moment it is established only under the additional condition $|E|=2$?}
If $A_E=\{2\}$, then by the Lemma~\ref{l:Kirch}, the filter $\F_{E}$ is generated by the base consisting of the sets $q\IZ^\bullet$ for an odd square-free number $q\in\IZ^\bullet$. Therefore $\F_E=\F_\infty$ by the Lemma~\ref{filter0}.
\smallskip	

$(2)\Rightarrow(3)$:	
 Assume that $|E|=2$ and $A_E=\{2\}$. By Lemma~\ref{2ae}, $E=\{\varepsilon2^a,\delta2^b\}$, where $a,b\in \omega$ and $\varepsilon,\delta\in\{-1,1\}$. Without loss of generality we can assume that $b\leq a$. By Lemma~\ref{2ae},  $\Pi_{\varepsilon2^a-\delta2^b}=\Pi_{ \varepsilon2^b(2^{a-b}-\delta/\varepsilon)}\subseteq\{2\}$. The last inclusion implies  that $a-b=1$ and  $\delta/\varepsilon=1$ or  $a-b=0$ and $\delta/\varepsilon=-1$. In  the first case the set $E$ equals $\{2^b,2^{b+1}\}$ or $\{-2^b,-2^{b+1}\}$, in the second case  $E=\{2^b,-2^{b}\}$.

$(3)\Rightarrow(2)$: The implication $(3)\Ra(2)$ follows from Lemma~\ref{2ae}.
\end{proof}

In the following lemmas by $\mathfrak F'$ we denote the set  of maximal elements of the poset $\mathfrak F\setminus\{\F_\infty\}$.

%By $2^{<\w}$ we shall denote the set $\{2^n:n\in\w\}$ of powers of $2$.

\begin{lemma}\label{l:max} For a nonempty finite subset $E\subseteq \IZ^\bullet $ the filter $\F_E$ belongs to the family $\mathfrak F'$ if and only if there exists an odd prime number $p\notin\Pi_E$ such that $A_E=\{2,p\}$.
\end{lemma}

\begin{proof} To prove the ``if'' part, assume that $A_E=\{2,p\}$ and $p\notin\Pi_E$ for some odd prime number $p$. By Lemma~\ref{l:2max}, $\F_E\ne\F_\infty$. To show that the filter $\F_E$ is maximal in $\mathfrak F\setminus\{\F_\infty\}$, take any finite set $F\subset \IZ^\bullet $ such that $\F_E\subseteq\F_F\ne\F_\infty$. By Lemmas~\ref{l:wo2} and \ref{l:2max}, $\{2\}\ne A_F\subseteq\{2,p\}$, $\Pi_F\subseteq \Pi_E\cup\{2\}=\{2\}$, and $\alpha_F{\restriction}A_F\setminus \Pi_E=\alpha_E{\restriction}A_F\setminus \Pi_E$. It follows that $A_F=\{2,p\}=A_E$, $\Pi_F\cup\{2\}=\Pi_E\cup\{2\}$ and $\alpha_F=\alpha_E$. Applying Lemma~\ref{l:wo2}, we conclude that $\F_E=\F_F$, which means that the filter $\F_E$ is a maximal element of the poset $\mathcal F\setminus\{\F_\infty\}$.
\smallskip

To prove the ``only if'' part, assume that $\F_E\in\mathfrak F'$. By Lemma~\ref{l:2max}, $A_E\ne\{2\}$ and hence there exists an odd prime number $p\in A_E$.
We claim that $p\notin\Pi_E$. To derive a contradiction, assume that $p\in \Pi_E$ and consider the sets $F=\{p,2p\}$ and $G=\{1,p,2p\}$. By Lemma~\ref{2ae}, $A_F=A_G=\{2,p\}$, $\Pi_F=\{p\}$, and $\Pi_G=\emptyset$. Taking into account that  $F\subset G$, $A_F=\{2,p\}\subseteq A_E$, $\Pi_F\setminus\{2\}=\{p\}\subseteq \Pi_E$ and $A_F\setminus\Pi_E\subseteq\{2\}$, we can apply Lemmas~\ref{l:wo2}, \ref{l:2max} and conclude that $\F_E\subseteq \F_F\subseteq\F_G\ne\F_\infty$. The maximality of $\F_E$ implies $\F_E=\F_F=\F_G$. By Lemma~\ref{l:wo2}, the equality $\F_G=\F_F$ implies $p\in\Pi_F\setminus\{2\}\subseteq\Pi_G=\emptyset$, which is a contradiction showing that $p\notin\Pi_E$.

Now consider the set $H=\{\alpha_E(p),p,2p\}$ and observe that $A_H=\{2,p\}$, $\Pi_H=\emptyset$  and $\alpha_H(p)=\alpha_E(p)$. Lemmas~\ref{l:wo2} and \ref{l:2max} guarantee that $\F_E\subseteq\F_H\ne\F_\infty$. By the maximality of $\F_E$, we have $\F_E=\F_H$. Applying Lemma~\ref{l:wo2} once more, we conclude that $A_E=A_H=\{2,p\}$.
\end{proof}

Lemma~\ref{l:max} implies the following description of the set $\mathfrak F'$.

\begin{lemma} $\mathfrak F'=\{\F_{\{a,p,2p\}}:p\in\Pi\setminus\{2\},\;\;a\in\{1,\dots,p-1\}\}$.
\end{lemma}

Let $\mathfrak F''$ be the set of maximal elements of the poset $\mathfrak F\setminus(\mathfrak F'\cup\{\mathcal F_\infty\})$

\begin{lemma}\label{efbis}
	 For a finite set $E\subset\IZ^\bullet$, the filter $\mathcal F_E$ belongs to the family $\mathfrak F''$ if and only if one of the following conditions holds:
\begin{enumerate}
\item there exists an odd prime number $p$ such that $p\in \Pi_E$ and $A_E=\{2,p\}$;
\item there are two distinct odd prime numbers $p,q$ such that $A_E=\{2,p,q\}$ and $\Pi_E\subseteq\{2\}$.
\end{enumerate}
\end{lemma}

\begin{proof} To prove the ``only if'' part, assume that $\F_E\in\mathfrak{F}''$. By Lemma~\ref{l:2max}, there is an odd prime number $p\in A_E$. If  $A_E=\{2,p\}$, then $p\in\Pi_E$ by Lemma~\ref{l:max}, and  condition (1) is satisfied. So, we assume that $\{2,p\}\ne A_E$ and find an odd prime number $q\in A_E\setminus\{2,p\}$. By Lemma~\ref{l:realization},  there is a number $x\in\IN$ such that for the set $F=\{x,pq,2pq\}$ we have $A_F=\{2,p,q\}$, $\Pi_F=\emptyset$, $\alpha_F(p)=\alpha_E(p)$ and $\alpha_F(q)=\alpha_E(q)$. Then $\F_E\subseteq \F_F$ by Lemma~\ref{l:wo2}, and $\F_F\in\mathfrak{F}\setminus(\mathfrak{F}'\cup\{\F_\infty\})$ by Lemma~\ref{l:max}. Now the maximality of the filter $\F_E$ implies that $\F_E=\F_F$ and hence $A_E=A_F=\{2,p,q\}$ and $\Pi_E\subseteq \Pi_F\cup\{2\}=\{2\}$, see Lemma~\ref{l:wo2}.
\smallskip

To prove the ``if'' part, we consider two cases. First we assume that $A_E=\{2,p\}$ and $p\in \Pi_E$ for some odd prime number $p$.	By Lemmas~\ref{l:2max}  and \ref{l:max}, $\F_{E}\in \mathfrak{F}\setminus(\{\F_\infty\}\cup \mathfrak{F}')$. To prove that $\F_E$ is a maximal element of 	$\mathfrak{F}\setminus(\{\F_\infty\}\cup \mathfrak{F}')$, take any finite set $F\subseteq\IZ^\bullet$ such that $\F_E\subseteq\F_F\in\mathfrak{F}\setminus(\{\F_\infty\}\cup \mathfrak{F}')$. Lemma~\ref{l:wo1} implies that $\min\{|E|,|F|\}\ge 2$ and then by Lemmas~\ref{l:wo2} and \ref{l:max}, we have $A_F=\{2,p\}$, $\Pi_F\setminus\{2\}\subseteq \{p\}$ and
	$\alpha_E{\restriction}A_F\setminus\{p\}=\alpha_F{\restriction}A_F\setminus\{p\}$. Now notice that $p \in \Pi_F$ since otherwise $\F_F\in\mathfrak{F}'$ by Lemma~\ref{l:max}. By using again Lemma~\ref{l:wo2}  we get $\F_F=\F_E$ which means that $\F_E\in\mathfrak{F}''$. 		
	
	Now assume that there are two distinct odd prime numbers $p,q$ such that $A_E=\{2,p,q\}$ and $\Pi_E\subseteq\{2\}$. By Lemmas~\ref{l:2max}  and \ref{l:max}, $\F_{E}\in \mathfrak{F}\setminus(\{\F_\infty\}\cup \mathfrak{F}')$. To prove that $\F_E$ is a maximal element of  $\mathfrak{F}\setminus(\{\F_\infty\}\cup \mathfrak{F}')$, take any finite set $F\subseteq\IZ^\bullet$ such that $\F_E\subseteq \F_F\in\mathfrak{F}\setminus(\{\F_\infty\}\cup \mathfrak{F}')$. Lemma~\ref{l:wo2} implies that $A_F\subseteq \{2,p,q\}$, $\Pi_F\subseteq \{2\}$ and
	$\alpha_E{\restriction}A_F\setminus\Pi_E=\alpha_F{\restriction}A_F\setminus\Pi_E$. Taking into account that $\F_F\notin\mathfrak F'\cup\{\F_\infty\}$ and $\Pi_F\subseteq\{2\}$, we can apply Lemmas~\ref{l:max}, \ref{l:2max} and conclude that $A_F=\{2,p,q\}$. We therefore know that $A_F=A_E$, $\Pi_E\cup\{2\}=\Pi_F\cup\{2\}$ and $\alpha_F{\restriction}A_E\setminus\Pi_F=\alpha_E{\restriction}A_E\setminus\Pi_F$. By Lemma~\ref{l:wo2}, $\F_E=\F_F$ and hence  $\F_E\in\mathfrak{F}''$.
\end{proof}

\begin{lemma}\label{l:p2p-fix} For any homeomorphism $h$ of the Kirch space and any odd prime number $p$ we have $$\tilde{h}(\F_{\{p,2p\}})=\F_{\{p,2p\}}.$$
%,\;\;\tilde{h}(\F_{\{1,p,2p\}})=\F_{\{1,p,2p\}}\quad\mbox{and}\quad \tilde{h}(\F_{\{2,p,2p\}})=\F_{\{2,p,2p\}}.$$
\end{lemma}

\begin{proof} By Proposition~\ref{p:iso}, the homeomorphism $h$ induces an order isomorphism $\tilde h$ of the superconnecting poset $\mathfrak F$ on the Kirch space. Then $\tilde h[\mathfrak F']=\mathfrak F'$ and $\tilde h[\mathfrak F'']=\mathfrak F''$.

By Lemmas~\ref{efbis} and \ref{l:realization}, $\mathfrak F''=\mathfrak F''_2\cup\mathfrak F''_3$ where
$$
\begin{aligned}
\mathfrak F''_2&=\{\F_{\{p,2p\}}:p\in\Pi\setminus\{2\}\}\quad\mbox{and}\\
\mathfrak F''_3&=\{\F_{\{x,pq,2pq\}}:p,q\in\Pi\setminus\{3\},\;x\in\{0,\dots,pq-1\}\setminus(p\IZ\cup q\IZ)\}.
\end{aligned}
$$
By Lemmas~\ref{l:wo2} and \ref{l:max}, for every filter $\F_{\{p,2p\}}\in\mathfrak F''_2$ the set
${\uparrow}\F_{\{p,2p\}}=\{\F\in \mathfrak F':\F_{\{p,2p\}}\subset\F_E\}$ coincides with the set $\{\F_{\{a,p,2p\}}:a\in\{1,\dots,p-1\}\}$ and hence has cardinality $p-1$.

On the other hand, for any filter $\F_{\{x,pq,2pq\}}\in\mathfrak F''_3$, the set ${\uparrow}\F_{\{x,pq,2pq\}}=\{\F\in \mathfrak F':\F_{\{x,pq,2pq\}}\subset\F\}$ coincides with the doubleton $\{\F_{\{x,p,2p\}},\F_{\{x,q,2q\}}\}$.

These order properties uniquely determine the filters $\F_{\{p,2p\}}$ for $p\in\Pi\setminus\{3\}$ and ensure that $\tilde h(\F_{\{p,2p\}})=\F_{\{p,2p\}}$ for every $p\in\Pi\setminus\{3\}$.

Next, observe that $\F_{\{3,6\}}$ is a unique element $\F$ of $\mathfrak F''$ such that ${\uparrow}\F\cap\bigcup_{p\in\Pi\setminus\{3\}}{\uparrow}\F_{\{p,2p\}}=\emptyset$. This uniqueness order property of $\F_{\{3,6\}}$ ensures that $\tilde h(\F_{\{3,6\}})=\F_{\{3,6\}}$.
\end{proof}

\begin{lemma}\label{lemmma} Let $E\subseteq \IZ^\bullet$ be a finite subset such that $A_E=\{2,p\}$ for some odd prime number $p\notin \Pi_E$. Then $A_{h[E]}=\{2,p\}$.
\end{lemma}

\begin{proof} By Lemma \ref{l:max}, $\F_E\in\mathfrak{F'}$. Consider the doubleton $\{p,2p\}$ which has $A_{\{p,2p\}}=\{2,p\}$ and $\Pi_{\{p,2p\}}=\{p\}$. By Lemma \ref{l:wo2}, $\mathcal{F}_{\{p,2p\}}\subseteq\F_E$ and by Lemma \ref{l:p2p-fix}, $\mathcal{F}_{\{p,2p\}}=\tilde{h}(\mathcal{F}_{\{p,2p\}})=\mathcal{F}_{\{h(p),h(2p)\}}\subseteq\mathcal{F}_{h[E]}$. By Lemma \ref{l:wo2}, $A_{h[E]}\subseteq A_{\{p,2p\}}=\{2,p\}$. By Lemma 8, $A_{h[E]}=\{2,p\}$.
\end{proof}

\begin{definition}  A homeomorphism $h$ of the Kirch space $(\IZ^\bullet,\tau_K)$ is called {\em positive} if $h(1)>0$.
\end{definition}

\begin{lemma}\label{l:2fix} Every positive homeomorphism $h$ of the Kirch space has $h(x)=x$ for any $x\in\{\pm2^n, n\in\omega\}$
\end{lemma}

\begin{proof} Consider the graph $\Gamma_2=(V_2,\mathcal E)$ with set of vertices $V_2=\{\pm2^n:n\in\omega\}$ and set of edges $\mathcal E=\big\{\{2^n,2^{n+1}\},\{-2^n,-2^{n+1}\},\{-2^n,2^n\}:n\in\omega\big\}$. %Obserev that the graph $\Gamma_2$ is rigid in the sense that it admits a unique isomorphism. This follows from the observation that $1$ is a unique vertex of order 1 in $\Gamma_2$. Consequently, $1$ is a fixed point of any by $2$ is a unique vertex connected by an edge with the

Observe that $1$ and $-1$ are the unique vertices of $\Gamma_2$ that have order 2. Any other vertices have order 3. This ensures that $h(1)=\pm1$. The positivity of $h$ implies that $h(1)=1$. Then $h(-1)=-1$, $h(2)=2$. Hence $h(\pm2^n)=\pm2^n$ for all $n\in\w$.
% By lemma 5, $\F_E= \F_{h[E]}$ if and only if $A_E= A_{h[E]},\;\; \Pi_E\setminus\{2\}=\Pi_{h[E]} \mbox{ \  and \ }\alpha_E{\restriction}A_{h[E]}\setminus\Pi_E=\alpha_{h[E]}{\restriction}A_{h[E]}\setminus\Pi_E.$
\end{proof}

Lemmas~\ref{l:2fix} and \ref{l:p2p-fix} imply

\begin{lemma}\label{l:12fix} For any positive homeomorphism $h$ of the Kirch space and any odd prime number $p$ we have $$\tilde{h}(\F_{\{1,p,2p\}})=\F_{\{1,p,2p\}}\quad\mbox{and}\quad \tilde{h}(\F_{\{2,p,2p\}})=\F_{\{2,p,2p\}}.$$
\end{lemma}

\begin{lemma}\label{ppix} For an integer number $x\in\IZ^\bullet\setminus\{-2,-1,1,2\}$ and an odd prime $p$, the following conditions are equivalent:
\begin{enumerate}
\item $p\in\Pi_x$;
\item $\F_{\{1,x\}}\subseteq \F_{\{1,p,2p\}}$ and $\F_{\{2,x\}}\subseteq \F_{\{2,p,2p\}}$.
\end{enumerate}
\end{lemma}

\begin{proof} If $p\in\Pi_x$, then $A_{\{1,p,2p\}}=\{2,p\}\subseteq A_{\{1,x\}}$, $\Pi_{\{1,x\}}\cup\{2\}=\{2\}=\Pi_{\{1,p,2p\}}\cup\{2\}$  and $\alpha_{\{1,x\}}(p)=1=\alpha_{\{1,p,2p\}}(p)$. By Lemma~\ref{l:wo2}, $\F_{\{1,x\}}\subseteq\F_{\{1,p,2p\}}$. By analogy we can prove that $\F_{\{2,x\}}\subseteq \F_{\{2,p,2p\}}$.

Conversely, assume $\F_{\{1,x\}}\subseteq \F_{\{1,p,2p\}}$ and $\F_{\{2,x\}}\subseteq \F_{\{2,p,2p\}}$. By Lemmas~\ref{l:wo2} and \ref{2ae}, we have
$$\{2,p\}=A_{\{1,p,2p\}}\subseteq A_{\{1,x\}}=\Pi_{x}\cup\Pi_{x-1}
\text{ and }\{2,p\}=A_{\{2,p,2p\}}\subseteq A_{\{2,x\}}=\{2\}\cup\Pi_x\cup\Pi_{x-2}$$
and hence $p\in (\Pi_x\cup\Pi_{x-1})\cap(\Pi_x \cup\Pi_{x-2})\setminus\{2\}\subseteq \Pi_x$.
\end{proof}

%\begin{lemma}\label{lst1} For every homeomorphism $h$ of the Kirch space and every $x\ge 3$ we have $\Pi_x\setminus\{2\}=\Pi_{h(x)}\setminus\{2\}$.
%\end{lemma}

Proposition~\ref{p:iso} and Lemmas~\ref{l:2fix}, \ref{l:12fix}, \ref{ppix} imply

\begin{lemma}\label{Pix} For every homeomorphism $h$ of the Kirch space and any number $x\in\IN$ we have $$\Pi_x\cup\{2\}=\Pi_{h(x)}\cup\{2\}.$$
\end{lemma}

%\begin{proof}
%{\color{red} Fix $n\in \IN$ and consider $x=pn$.
%Then 	$\F_{\{1,x\}}\subseteq \F_{\{1,p,2p\}}$ and $\F_{\{2,x\}}\subseteq \F_{\{2,p,2p\}}$
%	by  Lemma~\ref{ppix}.
%Lemmas 13 and 14 imply that
%	$$\F_{\{1,h(x)\}}=\widetilde{h}(\F_{\{1,x\}})\subseteq \widetilde{h}(\F_{\{1,p,2p\}})=\F_{\{1,p,2p\}}$$
%	and
%$$\F_{\{2,h(x)\}}=\widetilde{h}(\F_{\{2,x\}})\subseteq \widetilde{h}(\F_{\{2,p,2p\}})=\F_{\{2,p,2p\}}.$$
%Using again Lemma~\ref{ppix} we obtain that
%$p\in\Pi_{h(x)}$ what means $h(x)\in p\IN$.
%Thus $h(p\IN)\subseteq p\IN$ for each homeomorphism $h$ of the Kirch space. So, $h^{-1}(p\IN)\subseteq \IN$, and this implies that $h (p\IN) = p\IN$ for every  homeomorphism of the Kirch space.
%}\end{proof}

For every prime number $p$ consider the set $$V_p=\{\pm2^{n-1}p^m:n,m\in\IN\}$$ of numbers $x\in\IN$ such that $p\in\Pi_x\subseteq \{2,p\}$. Lemmas~\ref{l:2fix} and \ref{Pix} imply that $h[V_p]=V_p$ for every homeomorphism $h$ of the Kirch space.

Consider the graph $\Gamma_p=(V_p,\mathcal E_p)$ on the set $V_p$ with the set of edges $$\mathcal E_p:=\big\{E\in[V_p]^2:A_E=\{2,p\}\big\}.$$

\begin{lemma}\label{l:graph} For every prime number $p$ and every homeomorphism $h$ of the Kirch space, the restriction of $h$ to $V_p$ is an isomorphism of the graph $\Gamma_p$.
\end{lemma}

\begin{proof} Let $E\in \mathcal{E}_p$. Since $p\in \Pi_E$, we can apply Lemma~\ref{efbis} and conclude that $\F_E\in\mathfrak{F}''$.  Using fact that $\tilde{h}$ is isomorphism of $\mathfrak{F}$ we get $\F_{h[E]}=\tilde{h}(\F_E)\in\mathfrak{F}''$. Since $h[E]\subseteq h[V_p]=V_p$, we obtain $p\in \Pi_{h[E]}$. Using Lemma~\ref{efbis} once more, we obtain that $A_{h[E]}=\{2,p\}$, which means that $h[E]\in \mathcal{E}_p$. By analogical reasoning we can prove that $h^{-1}[E]\in \mathcal{E}_p$ for every $E\in \mathcal{E}_p$. This means that $h{\restriction}V_p$ is isomorphism of the graph $\Gamma_p$.
	\end{proof}

The structure of the graph $\Gamma_p$ depends on properties of the prime number $p$.

A prime number $p$ is called
\begin{itemize}
\item {\em Fermat prime} if $p=2^n+1$ for some $n\in\IN$;
\item {\em Mersenne prime} if $p=2^n-1$ for some $n\in\IN$;
\item {\em Fermat--Mersenne} if $p$ is Fermat prime or Mersenne prime.
\end{itemize}
It is known (and easy to see) that for any Fermat prime number $p=2^n+1$ the exponent $n$ is a power of $2$, and for any Mersenne prime number $p=2^n-1$ the power $n$ is a prime number. It is not known whether there are infinitely many Fermat--Mersenne prime numbers. All known Fermat prime numbers are the numbers $2^{2^n}+1$ for $0\le n\le 4$ (see {\tt oeis.org/A019434} in \cite{OEIS}). At the moment only 51 Mersenne prime numbers are known, see the sequence {\tt oeis.org/A000043} in \cite{OEIS}.

%\begin{lemma}\label{3}

%For every homeomorphism $h$ of the Kirch space,  $h(3)=3$.
%\end{lemma}

%\begin{proof} Let consider the set $\{-1,3\}$. By Lemma \ref{2ae}, $A_{\{-1,3\}}=\{2,3\}$ and $A_{h(\{-1,3\})}=A_{\{-1,h(3)\}}$. Again by Lemma \ref{2ae}, $h(3)=\pm2^a3^b, \ a\in\omega,b\in\omega\setminus\{0\}$. By Lemma \ref{lemmma}, $A_{\{-1,\pm2^a3^b\}}=\{2,3\}$.  Since $\Pi_{2^a3^b\pm1}\subseteq\{2,3\}$, it can be possible if and only if $a=0,b=1$.
%\end{proof}

\begin{lemma}\label{struct} Let $p$ be an odd prime number, $p\neq3$.
\begin{enumerate}
\item If $p=3$, then the set $\mathcal E_p$ of the edges of the graph $\Gamma_p$ coincides with the set of doubletons\\
$\{\varepsilon2^{a-1}3^b,\varepsilon2^{a-1}3^{b+1}\}$, $\{\varepsilon2^{a-1}3^b,2\varepsilon^{a-1}3^{b+2}\}$, $\{\varepsilon2^{a-1}3^b,\varepsilon2^{a}3^{b}\}$, $\{\varepsilon2^{a-1}3^b,\varepsilon2^{a+1}3^{b}\}$, $\{\varepsilon2^{a-1}3^{b+1},\varepsilon2^{a+1}3^b\}$,\\
$\{\varepsilon2^{a+1}3^{b},\varepsilon2^{a}3^{b+1}\}$, $\{\varepsilon2^{a+3}3^b,\varepsilon2^a3^{b+2}\},$ $  \{\varepsilon2^{a-1}3^b,-\varepsilon2^{a-1}3^{b+1}\}$, $\{\varepsilon2^{a-1}3^b,-\varepsilon2^{a}3^{b}\}$\\ $\{\varepsilon2^{a-1}3^b,-\varepsilon2^{a+2}3^{b}\}$, $\{\varepsilon2^{a-1}3^b,-\varepsilon2^{a-1}3^b\}$
for some $a,b\in\IN,\ \varepsilon \in\{-1,1\}$.
\item If $p=2^m+1>3$ is Fermat prime, then\newline $\mathcal E_p=\big\{\{\varepsilon2^{a-1}p^b,\varepsilon2^{a-1}p^{b+1}\}, \{\varepsilon2^{a-1}p^b,\varepsilon2^{a}p^b\}, \{\varepsilon2^{a-1}p^b,-\varepsilon2^{a+m-1}p^{b}\},$

     $ \{\varepsilon2^{a-1}p^b,-\varepsilon2^{a-1}p^b\}, \{\varepsilon2^{m+a-1}p^{b},\varepsilon2^{a-1}p^{b+1}\}:a,b\in\IN, \ \varepsilon \in\{-1,1\}\big\}$.
\item
If $p=2^m-1>3$ is Mersenne prime, then
 $\mathcal E_p=\big\{\{\varepsilon2^ap^b,\varepsilon2^{a-1}p^b\},\{\varepsilon2^{a-1}p^b,\varepsilon2^{m+a-1}p^b\}, $\break$ \{\varepsilon2^{a-1}p^{b+1},\varepsilon2^{m+a-1}p^b\},$ $\{\varepsilon2^{a-1}p^b,-\varepsilon2^{a-1}p^b\}, \{\varepsilon2^{a-1}p^b,-\varepsilon2^{a-1}p^{b+1}\}
 :a,b\in\IN, \ \varepsilon \in\{-1,1\}\big\}$.
\item If $p$ is not Fermat--Mersenne, then $\mathcal E_p=\big\{\{\varepsilon2^{a-1}p^b,-\varepsilon2^{a-1}p^b\}, \{\varepsilon2^{a-1}p^b,\varepsilon2^{a}p^b\}:a,b\in\IN, \ \varepsilon \in\{-1,1\}\big\}$.
\end{enumerate}
\end{lemma}

\begin{proof} 	Proof of Lemma~\ref{struct} in each of cases (1)--(4) will be similar. The edges of graph $\Gamma_p$ are  $2$-element subsets of set $V_p$ such that $A_E=\{2,p\}$. Since the vertices of graph $\Gamma_p$ are numbers of the form $\pm2^{n-1}p^m$, where $n,m\in\IN$, we can apply Lemma~\ref{2ae} and conclude that a doubleton $\{x,y\}\subset V_p$ belongs to $\mathcal E_p$ if and only if $\{2,p\}= \Pi_x\cup\Pi_y\cup\Pi_{x-y}$. In subsequent proofs, we will intensively use the Mih\u ailescu Theorem~\ref{Mihailescu} saying that $2^3,3^2$ is a unique pair of consecutive powers.
%\smallskip	
	
1. First we consider the case of $p=3$. It is easy to see that the doubletons $\{x,y\}$ written in the statement (1) have $\Pi_x\cup\Pi_y\cup\Pi_{x-y}\subseteq\{2,3\}$, which implies that $\{x,y\}\in\mathcal E_3$. It remains to show that every doubleton $\{x,y\}\in\mathcal E_3$ is of the form indicated in the statement (1). Write $\{x,y\}$ as $\{\varepsilon 2^{a-1}3^b,\delta2^{c-1}3^d\}$ for some $a,b,c,d\in\IN$, $\varepsilon,\delta \in\{-1,1\}$ such that $2^{a-1}3^b\le 2^{c-1}3^d$.

If $a=c$ and $b=d$, then $\varepsilon\ne\delta$ and $\{x,y\}=\{\varepsilon 2^{a-1}3^b,-\varepsilon 2^{a-1}3^b\}$.

If $a=c$, then $b\le d$ and the inclusion $\Pi_{x-y}\subseteq\{2,3\}$ implies that $\Pi_{3^{d-b}-\varepsilon/\delta}\subseteq\{2,3\}$ and hence $3^{d-b}-\varepsilon/\delta$ is a power of $2$. If $\varepsilon/\delta=1$ then by the Mih\u ailescu Theorem~\ref{Mihailescu}, $d-b\in\{1,2\}$, which means that $\{x,y\}$ is equal to $\{\varepsilon2^{a-1}3^b,\varepsilon2^{a-1}3^{b+1}\}$ or $\{\varepsilon2^{a-1}3^b,\varepsilon2^{a-1}3^{b+2}\}$. If $\varepsilon/\delta=-1$ then  by the Mih\u ailescu Theorem~\ref{Mihailescu}, $d-b\in\{0,1\}$, which means that $\{x,y\}$ is equal to  $\{\varepsilon2^{a-1}3^b,-\varepsilon2^{a-1}3^{b}\}$ or $\{\varepsilon2^{a-1}3^b,-\varepsilon2^{a-1}3^{b+1}\}$.

If $b=d$, then $a\le c$ and the inclusion  $\Pi_{x-y}\subseteq\{2,3\}$ implies that $\Pi_{2^{c-a}-\varepsilon/\delta}\subseteq\{2,3\}$ and hence $2^{c-a}-\varepsilon/\delta$ is either 2 or a power of $3$. If $\varepsilon/\delta=1$ then by the Mih\u ailescu Theorem~\ref{Mihailescu}, $c-a\in\{1,2\}$, which means that $\{x,y\}$ is equal to $\{\varepsilon2^{a-1}3^b,\varepsilon2^{a}3^{b}\}$ or $\{\varepsilon2^{a-1}3^b,\varepsilon2^{a+1}3^{b}\}$. If $\varepsilon/\delta=-1$ then by the Mih\u ailescu Theorem~\ref{Mihailescu}, $c-a\in\{0,1,3\}$, which means that $\{x,y\}$ is equal to  $\{\varepsilon2^{a-1}3^b,-\varepsilon2^{a-1}3^{b}\}$,  $\{\varepsilon2^{a-1}3^b,-\varepsilon2^{a}3^{b}\}$ or $\{\varepsilon2^{a-1}3^b,-\varepsilon2^{a+2}3^{b}\}$.

%If $a=c$, $b=d$ and $\varepsilon=-\delta$ then $\Pi_{\varepsilon2^{a-1}3^b-\delta2^{a-1}3^{b}}=\Pi_{\varepsilon2^{a}3^b}\subset\{2,3\}$. In this case $\{x,y\}=\{\varepsilon2^{a-1}3^b,-\varepsilon2^{a-1}3^b\}$.

So, we assume that $a\ne c$ and $b\ne d$. In this case we should consider four subcases.

If $a<c$ and $b<d$, then $\Pi_{x-y}\subseteq\{2,3\}$ implies that each prime divisor of $2^{c-a}3^{d-b}-\varepsilon/\delta$ is equal to $2$ or $3$, which is not possible.

If  $a<c$ and $b>d$, then $\Pi_{x-y}\subseteq\{2,3\}$ and $2^{a-1}3^b\le 2^{c-1}3^d$ imply that $2^{c-a}-(\varepsilon/\delta)3^{b-d}=1$ which implies that $\varepsilon=\delta$. Hence $c-a=2$ and $b-d=1$ by the Mih\u ailescu Theorem~\ref{Mihailescu}.  In this case $\{x,y\}=\{\varepsilon 2^{a-1}3^{d+1},\varepsilon 2^{a+1}3^d\}$.

If $a>c$ and $b<d$, then  $\Pi_{x-y}\subseteq\{2,3\}$ and $2^{a-1}3^b\le 2^{c-1}3^d$ imply that $3^{d-b}-(\varepsilon/\delta)2^{a-c}=1$. This implies that  $\varepsilon/\delta=1$ and hence $\langle d-b,a-c\rangle\in\{\langle 1,1\rangle,\langle 2,3\rangle\}$ by the Mih\u ailescu Theorem~\ref{Mihailescu}.  In this case $\{x,y\}$ is equal to $\{2^{c+1}3^{b},2^{c}3^{b+1}\}$ or $\{2^{c+3}3^b,2^c3^{b+2}\}$.

The subcase $a>c$ and $b>d$ is forbidden by the inequality $2^{a-1}3^b\le 2^{c-1}3^d$.
\smallskip

\begin{figure}
$$\xymatrix@C50pt{
\vdots\ar@{-}@/_14pt/[dddddddd]&\vdots\ar@{-}@/_14pt/[dddddddd]&\vdots\ar@{-}[dddrr]\ar@{-}@/_14pt/[dddddddd]&\vdots\ar@{-}@/_14pt/[dddddddd]\\
2^3{\cdot}3\ar@{-}@/_14pt/[dddddddd]\ar@{-}[dr]\ar@{-}@/_32pt/[ddddddd]\ar@{-}[dddddddr]\ar@{-}[d]\ar@{-}@/^15pt/[rr]\ar@{-}[r]&2^3{\cdot}3^2\ar@{-}@/_14pt/[dddddddd]\ar@{-}[dr]\ar@{-}@/_32pt/[ddddddd]\ar@{-}[dddddddr]\ar@{-}[r]\ar@{-}[d]\ar@{-}@/^15pt/[rr]&2^3{\cdot}3^3\ar@{-}@/_14pt/[dddddddd]\ar@{-}[dr]\ar@{-}[dddrr]\ar@{-}@/_32pt/[ddddddd]\ar@{-}[dddddddr]\ar@{-}[r]\ar@{-}[d]\ar@{-}@/^15pt/[rr]&2^3{\cdot}3^4\ar@{-}@/_14pt/[dddddddd]\ar@{-}[dr]\ar@{-}@/_32pt/[ddddddd]\ar@{-}[r]\ar@{-}[d]&\dots\\
2^2{\cdot}3\ar@{-}@/_12pt/[dddddd]\ar@{-}[dr]\ar@{-}[ddr]\ar@{-}@/_21pt/[ddddd]\ar@{-}[dddddr]\ar@{-}@/_15pt/[uu]\ar@{-}[d]\ar@{-}@/^15pt/[rr]\ar@{-}[r]&2^2{\cdot}3^2\ar@{-}@/_12pt/[dddddd]\ar@{-}[dr]\ar@{-}@/_21pt/[ddddd]\ar@{-}[dddddr]\ar@{-}[uul]\ar@{-}@/_15pt/[uu]\ar@{-}[r]\ar@{-}[d]\ar@{-}@/^15pt/[rr]&2^2{\cdot}3^3\ar@{-}@/_12pt/[dddddd]\ar@{-}[dr]\ar@{-}@/_21pt/[ddddd]\ar@{-}[dddddr]\ar@{-}[uul]\ar@{-}@/_15pt/[uu]\ar@{-}[r]\ar@{-}[d]\ar@{-}@/^15pt/[rr]&2^2{\cdot}3^4\ar@{-}@/_12pt/[dddddd]\ar@{-}[dr]\ar@{-}@/_21pt/[ddddd]\ar@{-}[uul]\ar@{-}@/_15pt/[uu]\ar@{-}[r]\ar@{-}[d]&\dots\\
2{\cdot}3\ar@{-}@/_10pt/[dddd]\ar@{-}[dr]\ar@{-}@/_18pt/[ddd]\ar@{-}@/^22pt/[dddddd]\ar@{-}[dddr]\ar@{-}@/_15pt/[uu]\ar@{-}[d]\ar@{-}@/^15pt/[rr]\ar@{-}[r]&2{\cdot}3^2\ar@{-}@/_10pt/[dddd]\ar@{-}[dr]\ar@{-}@/_18pt/[ddd]\ar@{-}@/^22pt/[dddddd]\ar@{-}[dddr]\ar@{-}[uul]\ar@{-}@/_15pt/[uu]\ar@{-}[r]\ar@{-}[d]\ar@{-}@/^15pt/[rr]&2{\cdot}3^3\ar@{-}@/_10pt/[dddd]\ar@{-}[dr]\ar@{-}@/_18pt/[ddd]\ar@{-}@/^22pt/[dddddd]\ar@{-}[dddr]\ar@{-}[uuull]\ar@{-}[uul]\ar@{-}@/_15pt/[uu]\ar@{-}[r]\ar@{-}[d]\ar@{-}@/^15pt/[rr]&2{\cdot}3^4\ar@{-}@/_10pt/[dddd]\ar@{-}[dr]\ar@{-}@/_18pt/[ddd]\ar@{-}@/^22pt/[dddddd]\ar@{-}[uuull]\ar@{-}[uul]\ar@{-}@/_15pt/[uu]\ar@{-}[r]\ar@{-}[d]&\dots\\
3\ar@{-}@/_8pt/[dd]\ar@{-}@/^20pt/[dddd]\ar@{-}[rd]\ar@{-}@/_15pt/[uu]\ar@{-}[d]\ar@{-}@/^15pt/[rr]\ar@{-}[r]&3^2\ar@{-}@/_8pt/[dd]\ar@{-}@/^20pt/[dddd]\ar@{-}[rd]\ar@{-}@/_15pt/[uu]\ar@{-}[r]\ar@{-}[d]\ar@{-}@/^15pt/[rr]&3^3\ar@{-}@/_8pt/[dd]\ar@{-}@/^20pt/[dddd]\ar@{-}[rd]\ar@{-}[uul]\ar@{-}[uuull]\ar@{-}@/_15pt/[uu]\ar@{-}[r]\ar@{-}[d]\ar@{-}@/^15pt/[rr]&3^4\ar@{-}@/_8pt/[dd]\ar@{-}@/^20pt/[dddd]\ar@{-}[uul]\ar@{-}[uuull]\ar@{-}@/_15pt/[uu]\ar@{-}[r]\ar@{-}[d]&\dots\ar@{-}[dl]\\
-3\ar@{-}@/_8pt/[uu]\ar@{-}@/_20pt/[uuuu]\ar@{-}[ur]\ar@{-}[d]\ar@{-}@/^15pt/[dd]\ar@{-}@/_15pt/[rr]\ar@{-}[r]&-3^2\ar@{-}@/_8pt/[uu]\ar@{-}@/_20pt/[uuuu]\ar@{-}[ur]\ar@{-}[r]\ar@{-}[d]\ar@{-}@/_15pt/[rr]\ar@{-}[ddl]\ar@{-}@/^15pt/[dd]&-3^3\ar@{-}@/_8pt/[uu]\ar@{-}@/_20pt/[uuuu]\ar@{-}[ur]\ar@{-}[ddl]\ar@{-}[dddll]\ar@{-}@/^15pt/[dd]\ar@{-}[r]\ar@{-}[d]\ar@{-}@/_15pt/[rr]&-3^4\ar@{-}@/_8pt/[uu]\ar@{-}@/_20pt/[uuuu]\ar@{-}[ddl]\ar@{-}@/^15pt/[dd]\ar@{-}[r]\ar@{-}[d]\ar@{-}[dddll]&\dots\ar@{-}[dddll]\ar@{-}[ul]\\
-2{\cdot}3\ar@{-}@/_10pt/[uuuu]\ar@{-}[ur]\ar@{-}@/_22pt/[uuuuuu]\ar@{-}[uuur]\ar@{-}@/^15pt/[dd]\ar@{-}[d]\ar@{-}@/_15pt/[rr]\ar@{-}[r]&-2{\cdot}3^2\ar@{-}@/_10pt/[uuuu]\ar@{-}[ur]\ar@{-}@/_22pt/[uuuuuu]\ar@{-}[uuur]\ar@{-}[ddl]\ar@{-}@/^15pt/[dd]\ar@{-}[r]\ar@{-}[d]\ar@{-}@/_15pt/[rr]&-2{\cdot}3^3\ar@{-}@/_10pt/[uuuu]\ar@{-}[ur]\ar@{-}@/_22pt/[uuuuuu]\ar@{-}[uuur]\ar@{-}[dddll]\ar@{-}[ddl]\ar@{-}@/^15pt/[dd]\ar@{-}[r]\ar@{-}[d]\ar@{-}@/_15pt/[rr]&-2{\cdot}3^4\ar@{-}@/_10pt/[uuuu]\ar@{-}[ur]\ar@{-}@/_22pt/[uuuuuu]\ar@{-}[dddll]\ar@{-}[ddl]\ar@{-}@/^15pt/[dd]\ar@{-}[r]\ar@{-}[d]&\dots\ar@{-}[dddll]\\
-2^2{\cdot}3\ar@{-}@/_12pt/[uuuuuu]\ar@{-}[ur]\ar@{-}[uuuuur]\ar@{-}@/^15pt/[dd]\ar@{-}[d]\ar@{-}@/_15pt/[rr]\ar@{-}[r]&-2^2{\cdot}3^2\ar@{-}@/_12pt/[uuuuuu]\ar@{-}[ur]\ar@{-}[uuuuur]\ar@{-}[ddl]\ar@{-}@/^15pt/[dd]\ar@{-}[r]\ar@{-}[d]\ar@{-}@/_15pt/[rr]&-2^2{\cdot}3^3\ar@{-}@/_12pt/[uuuuuu]\ar@{-}[ur]\ar@{-}[uuuuur]\ar@{-}[ddl]\ar@{-}@/^15pt/[dd]\ar@{-}[r]\ar@{-}[d]\ar@{-}@/_15pt/[rr]&-2^2{\cdot}3^4\ar@{-}@/_12pt/[uuuuuu]\ar@{-}[ur]\ar@{-}[ddl]\ar@{-}@/^15pt/[dd]\ar@{-}[r]\ar@{-}[d]&\dots\\
-2^3{\cdot}3\ar@{-}@/_15pt/[uuuuuuuu]\ar@{-}[ur]\ar@{-}[uuuuuuur]\ar@{-}[d]\ar@{-}@/_15pt/[rr]\ar@{-}[r]&-2^3{\cdot}3^2\ar@{-}@/_15pt/[uuuuuuuu]\ar@{-}[ur]\ar@{-}[uuuuuuur]\ar@{-}[r]\ar@{-}[d]\ar@{-}@/_15pt/[rr]&-2^3\ar@{-}@/_15pt/[uuuuuuuu]\ar@{-}[ur]\ar@{-}[uuuuuuur]{\cdot}3^3\ar@{-}[r]\ar@{-}[d]\ar@{-}@/_15pt/[rr]&-2^3\ar@{-}@/_15pt/[uuuuuuuu]\ar@{-}[ur]{\cdot}3^4\ar@{-}[r]\ar@{-}[d]&\dots\\
\vdots&\vdots&\vdots&\vdots\\
}
$$
\caption{The graph $\Gamma_3$}\label{fig3}
\end{figure}
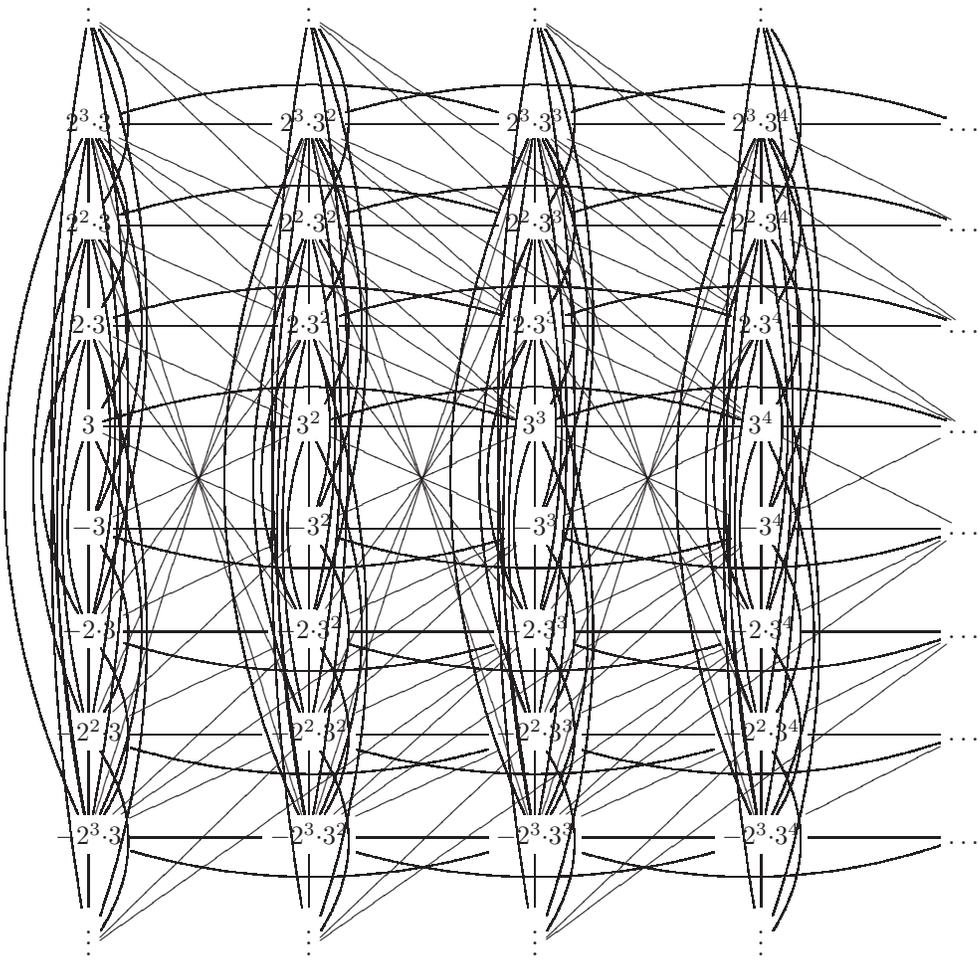

2. Assume that  $p=2^m+1>3$ is a Fermat prime. In this case $m>1$. Since $p>3$, $p$ is not Mersenne prime. It is easy to check that every doubleton $\{x,y\}\in\big\{\{\varepsilon2^{a-1}p^b,\varepsilon2^{a-1}p^{b+1}\}, \{\varepsilon2^{a-1}p^b,\varepsilon2^{a}p^b\},  \{\varepsilon2^{a-1}p^b,-\varepsilon2^{a+m-1}p^{b}\},$\break$ \{\varepsilon2^{a-1}p^b,-\varepsilon2^{a-1}p^b\}$,
$\{\varepsilon2^{m+a-1}p^{b},\varepsilon2^{a-1}p^{b+1}\}:a,b\in\IN, \ \varepsilon \in\{-1,1\}\big\}$ has $A_{\{x,y\}}= \Pi_x\cup\Pi_y\cup\Pi_{x-y}=\{2,p\}$ and hence $\{x,y\}\in\mathcal E_p$.

Now assume that $\{x,y\}\in\mathcal E_p$ is an edge of the graph $\Gamma_p$. Then $\Pi_x\cup\Pi_y\cup\Pi_{x-y}=A_{\{x,y\}}=\{2,p\}$ and $\{x,y\}$ can be written as $\{\varepsilon2^{a-1}p^b,\delta2^{c-1}p^d\}$ for some $a,b,c,d\in\IN$, $\varepsilon,\delta \in\{-1,1\}$ with $2^{a-1}p^b\le 2^{c-1}p^d$.

If $a=c$, $b=d$ and $\varepsilon=-\delta$ then $\Pi_{\varepsilon2^{a-1}p^b-\delta2^{a-1}p^{b}}=\Pi_{\varepsilon2^{a}p^b}\subset\{2,p\}$. In this case $\{x,y\}=\{\varepsilon2^{a-1}p^b,-\varepsilon2^{a-1}p^b\}$.

If $a=c$, then $b\le d$ and the inclusion $\Pi_{x-y}\subseteq\{2,p\}$ implies that $\Pi_{p^{d-b}-\varepsilon/\delta}\subseteq\{2,p\}$ and hence $p^{d-b}-\varepsilon/\delta$ is a power of $2$. By the Mih\u ailescu Theorem~\ref{Mihailescu}, $d-b\in\{0,1\}$.  If $d-b=0$, then $\varepsilon=-\delta$ and $\{x,y\}=\{\varepsilon 2^{a-1}p^b,-\varepsilon 2^{a-1}p^b\}$ by the preceding case. So, we assume that $d-b=1$. Since $p$ is not Mersenne prime, we conclude that $\varepsilon=\delta$, and hence  $\{x,y\}$ is equal to $\{ \varepsilon2^{a-1}p^b,\varepsilon2^{a-1}p^{b+1}\}$.

If $b=d$, then $a\le c$ and the inclusion  $\Pi_{x-y}\subseteq\{2,p\}$ implies that $\Pi_{2^{c-a}-\varepsilon/\delta}\subseteq\{2,p\}$ and hence $2^{c-a}-\varepsilon/\delta$ is a power of $p$. By the Mih\u ailescu Theorem~\ref{Mihailescu}, $2^{c-a}-\varepsilon/\delta\in\{1,p\}=\{1,2^m+1\}$. If $\varepsilon=\delta$ then $c-a=1$, which means that $\{x,y\}$ is equal to $\{\varepsilon2^{a-1}p^b,\varepsilon2^{a}p^{b}\}$. If $\varepsilon=-\delta$ then $c-a=m$ and $\{x,y\}=\{\varepsilon2^{a-1}p^b,-\varepsilon2^{a+m-1}p^{b}\}$.

So, we assume that $a\ne c$ and $b\ne d$. In this case we should consider four subcases.

If $a<c$ and $b<d$, then $\Pi_{x-y}\subseteq\{2,p\}$ implies that each prime divisor of $2^{c-a}p^{d-b}-\varepsilon/\delta$ is equal to $2$ or $p$, which is not possible.

If $a<c$ and $b>d$, then $\Pi_{x-y}\subseteq\{2,p\}$ implies that $2^{c-a}-(\varepsilon/\delta) p^{b-d}=1$ and hence $\varepsilon=\delta$. In this case the Mih\u ailescu Theorem~\ref{Mihailescu} ensures that $b-d=1$ and hence $2^{c-a}=p+1=2^m+2$ which is not possible (as $m>1$).

If $a>c$ and $b<d$, then $\Pi_{x-y}\subseteq\{2,p\}$ implies that $p^{d-b}-(\varepsilon/\delta)2^{a-c}=1$, which implies that $\varepsilon=\delta$. The Mih\u ailescu Theorem~\ref{Mihailescu} implies that $d-b=1$ and hence $2^{a-c}=p-1=2^m$ and $a-c=m$. In this case $\{x,y\}=\{\varepsilon2^{c+m-1}2^b,\varepsilon2^{c-1}p^{b+1}\}$.

The subcase $a>c$ and $b>d$ is forbidden by the inequality $2^{a-1}p^b\le 2^{c-1}p^d$.
%\smallskip

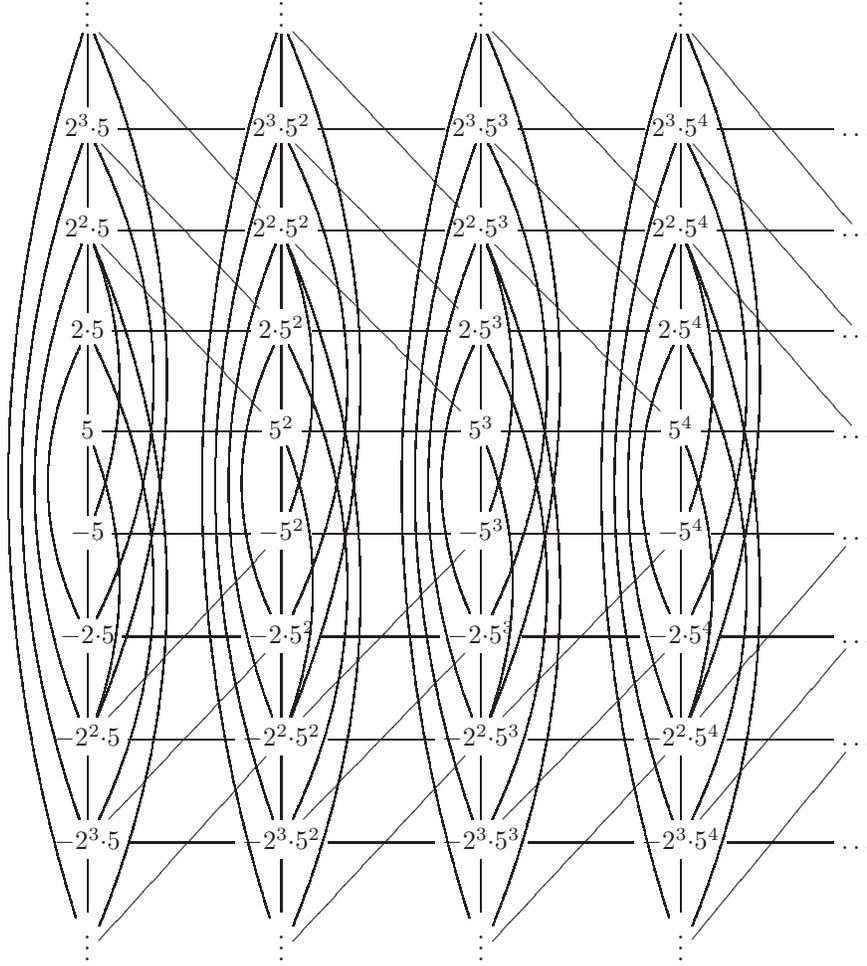
\begin{figure}
$$
\xymatrix@C40pt{
\vdots\ar@{-}@/^30pt/[ddddddd]\ar@{-}@/_30pt/[ddddddddd]\ar@{-}[d]\ar@{-}[rdd]&\vdots\ar@{-}@/_30pt/[ddddddddd]\ar@{-}@/^30pt/[ddddddd]\ar@{-}[d]\ar@{-}[rdd]&\ar@{-}@/_30pt/[ddddddddd]\ar@{-}@/^30pt/[ddddddd]\vdots\ar@{-}[d]\ar@{-}[rdd]&\vdots\ar@{-}@/_30pt/[ddddddddd]\ar@{-}@/^30pt/[ddddddd]\ar@{-}[d]\ar@{-}[rdd]&\\
2^3{\cdot}5\ar@{-}[r]\ar@{-}[d]\ar@{-}[rdd]\ar@{-}@/_25pt/[ddddddd]\ar@{-}@/^25pt/[ddddd]&2^3{\cdot}5^2\ar@{-}[r]\ar@{-}[d]\ar@{-}[rdd]\ar@{-}@/^25pt/[ddddd]\ar@{-}@/_25pt/[ddddddd]&2^3{\cdot}5^3\ar@{-}[r]\ar@{-}[d]\ar@{-}@/^25pt/[ddddd]\ar@{-}[rdd]\ar@{-}@/_25pt/[ddddddd]&2^3{\cdot}5^4\ar@{-}[r]\ar@{-}[d]\ar@{-}@/^25pt/[ddddd]\ar@{-}[rdd]\ar@{-}@/_25pt/[ddddddd]&\dots\\
2^2{\cdot}5\ar@{-}@/^12pt/[ddd]\ar@{-}[r]\ar@{-}[d]\ar@{-}[rdd]\ar@{-}@/^30pt/[ddddddd]\ar@{-}@/_20pt/[ddddd]&2^2{\cdot}5^2\ar@{-}@/^12pt/[ddd]\ar@{-}[r]\ar@{-}[d]\ar@{-}[rdd]\ar@{-}@/_20pt/[ddddd]\ar@{-}@/^30pt/[ddddddd]&2^2{\cdot}5^3\ar@{-}@/^12pt/[ddd]\ar@{-}[r]\ar@{-}[d]\ar@{-}[rdd]\ar@{-}@/_20pt/[ddddd]\ar@{-}@/^30pt/[ddddddd]&2^2{\cdot}5^4\ar@{-}@/^12pt/[ddd]\ar@{-}[r]\ar@{-}[d]\ar@{-}[rdd]\ar@{-}@/^30pt/[ddddddd]\ar@{-}@/_20pt/[ddddd]&\dots\\
2{\cdot}5\ar@{-}[r]\ar@{-}[d]\ar@{-}@/^25pt/[ddddd]\ar@{-}@/_15pt/[ddd]&2{\cdot}5^2\ar@{-}[r]\ar@{-}[d]\ar@{-}@/^25pt/[ddddd]\ar@{-}@/_15pt/[ddd]&2{\cdot}5^3\ar@{-}[r]\ar@{-}[d]\ar@{-}@/^25pt/[ddddd]\ar@{-}@/_15pt/[ddd]&2{\cdot}5^4\ar@{-}[r]\ar@{-}[d]\ar@{-}@/^25pt/[ddddd]\ar@{-}@/_15pt/[ddd]&\dots\\
5\ar@{-}[r]\ar@{-}[d]\ar@{-}@/^12pt/[ddd]&5^2\ar@{-}[r]\ar@{-}[d]\ar@{-}@/^12pt/[ddd]&5^3\ar@{-}[r]\ar@{-}[d]\ar@{-}@/^12pt/[ddd]&5^4\ar@{-}[r]\ar@{-}[d]\ar@{-}@/^12pt/[ddd]&\dots\\
-5\ar@{-}[r]\ar@{-}[d]&-5^2\ar@{-}[r]\ar@{-}[d]\ar@{-}[ldd]&-5^3\ar@{-}[r]\ar@{-}[d]\ar@{-}[ldd]&-5^4\ar@{-}[r]\ar@{-}[d]\ar@{-}[ldd]&\dots\ar@{-}[ldd]\\
-2{\cdot}5\ar@{-}[r]\ar@{-}[d]&-2{\cdot}5^2\ar@{-}[r]\ar@{-}[d]\ar@{-}[ldd]&-2{\cdot}5^3\ar@{-}[r]\ar@{-}[d]\ar@{-}[ldd]&-2{\cdot}5^4\ar@{-}[r]\ar@{-}[d]\ar@{-}[ldd]&\dots\ar@{-}[ldd]\\
-2^2{\cdot}5\ar@{-}[r]\ar@{-}[d]&-2^2{\cdot}5^2\ar@{-}[r]\ar@{-}[d]\ar@{-}[ldd]&-2^2{\cdot}5^3\ar@{-}[r]\ar@{-}[d]\ar@{-}[ldd]&-2^2{\cdot}5^4\ar@{-}[r]\ar@{-}[d]\ar@{-}[ldd]&\dots\ar@{-}[ldd]\\
-2^3{\cdot}5\ar@{-}[r]\ar@{-}[d]&-2^3{\cdot}5^2\ar@{-}[r]\ar@{-}[d]&-2^3{\cdot}5^3\ar@{-}[r]\ar@{-}[d]&-2^3{\cdot}5^4\ar@{-}[r]\ar@{-}[d]&\dots\\
\vdots&\vdots&\vdots&\vdots\\
%&&&&
}
$$
\caption{The graph $\Gamma_5$}\label{fig5}
\end{figure}

3. Assume that  $p=2^m-1>3$ is a Mersenne prime. In this case $m>2$ and $p$ is not Fermat.	It is easy to check that every doubleton $\{x,y\}\in\big\{\{\varepsilon2^ap^b,\varepsilon2^{a-1}p^b\},\{\varepsilon2^{a-1}p^b,\varepsilon2^{m+a-1}p^b\}, \{\varepsilon2^{a-1}p^{b+1},\varepsilon2^{m+a-1}p^b\},$ \break$\{\varepsilon2^{a-1}p^b,-\varepsilon2^{a-1}p^b\}, \{\varepsilon2^{a-1}p^b,-\varepsilon2^{a-1}p^{b+1}\} :a,b\in\IN, \  \varepsilon \in\{-1,1\} \big\}$ has $A_{\{x,y\}}= \Pi_x\cup\Pi_y\cup\Pi_{x-y}=\{2,p\}$ and hence $\{x,y\}\in\mathcal E_p$.

Now assume that $\{x,y\}\in\mathcal E_p$ is an edge of the graph $\Gamma_p$. Then $\Pi_x\cup\Pi_y\cup\Pi_{x-y}=A_{\{x,y\}}=\{2,p\}$ and $\{x,y\}$ can be written as $\{\varepsilon2^{a-1}p^b,\delta2^{c-1}p^d\}$ for some $a,b,c,d\in\IN$, $\varepsilon,\delta \in\{-1,1\}$ with $2^{a-1}p^b\le 2^{c-1}p^d$.

If $a=c$, $b=d$, then $\varepsilon=-\delta$ and  $\{x,y\}=\{\varepsilon2^{a-1}p^b,-\varepsilon2^{a-1}p^b\}$.

If $a=c$, then $b\le d$ and the inclusion $\Pi_{x-y}\subseteq\{2,p\}$ implies that $\Pi_{p^{d-b}-\varepsilon/\delta}\subseteq\{2,p\}$ and hence $p^{d-b}-\varepsilon/\delta$ is a power of $2$. By the Mih\u ailescu Theorem~\ref{Mihailescu}, $d-b\in\{0,1\}$. If $d-b=0$, then $\{x,y\}=\{\varepsilon2^{a-1}p^b,-\varepsilon2^{a-1}p^b\}$ by the preceding case. So, we assume that $d-b=1$. If $\varepsilon=\delta$, then  $p^{d-b}-\varepsilon/\delta=p-1=2^m-2$ is a power of $2$, which is not true as $m>2$. Therefore $\varepsilon=-\delta$ and $\{x,y\}$ is equal to $\{\varepsilon2^{a-1}p^b,-\varepsilon2^{a-1}p^{b+1}\}$

If $b=d$, then $a\le c$ and the inclusion  $\Pi_{x-y}\subseteq\{2,p\}$ implies that $\Pi_{2^{c-a}-\varepsilon/\delta}\subseteq\{2,p\}$ and hence $2^{c-a}-\varepsilon/\delta$ is a power of $p$. By the Mih\u ailescu Theorem~\ref{Mihailescu}, $2^{c-a}-\varepsilon/\delta\in\{1,p\}=\{1,2^m-1\}$, which implies that $\varepsilon=\delta$ and $c-a\in\{1,m\}$. Therefore $\{x,y\}$ is equal to $\{\varepsilon2^{a-1}p^b,\varepsilon2^{a}p^{b}\}$ or $\{\varepsilon2^{a-1}p^b,\varepsilon2^{m+a-1}p^b\}$.

So, we assume that $a\ne c$ and $b\ne d$. By analogy with the case of Fermat primes, we can show that the subcases ($a<c$ and $b<d$) and ($a>c$ and $b>d$) are impossible. 	

If $a<c$ and $b>d$, then $\Pi_{x-y}\subseteq\{2,p\}$ implies that $2^{c-a}-(\varepsilon/\delta) p^{b-d}=1$, and hence $\varepsilon/\delta=1$. Then the Mih\u ailescu Theorem~\ref{Mihailescu} ensures that $b-d=1$ and hence $2^{c-a}=p+1=2^m$ and $c-a=m$. In this case $\{x,y\}=\{\varepsilon2^{a-1}p^{d+1},\varepsilon2^{a+m-1}p^{d}\}$.

If $a>c$ and $b<d$, then $\Pi_{x-y}\subseteq\{2,p\}$ implies that $p^{d-b}-(\varepsilon/\delta)2^{a-c}=1$ and hence $\varepsilon/\delta=1$. Then Mih\u ailescu Theorem~\ref{Mihailescu} implies that $d-b=1$ and hence $2^{a-c}=p-1=2^m-2$, which is not possible as $m>2$.
%\smallskip

\begin{figure}
$$
\xymatrix@C60pt{
\vdots\ar@{-}@/_30pt/[ddddddddd]\ar@{-}[d]\ar@{-}[dddr]\ar@{-}@/^20pt/[ddd]&\ar@{-}@/_30pt/[ddddddddd]\vdots\ar@{-}[d]\ar@{-}@/^20pt/[ddd]\ar@{-}[dddr]&\ar@{-}@/_30pt/[ddddddddd]\vdots\ar@{-}[d]\ar@{-}@/^20pt/[ddd]\ar@{-}[dddr]&\ar@{-}@/_30pt/[ddddddddd]\vdots\ar@{-}[d]\ar@{-}@/^20pt/[ddd]\ar@{-}[dddr]\\
2^3{\cdot}7\ar@{-}[dddddddr]\ar@{-}@/^20pt/[ddd]\ar@{-}[d]\ar@{-}[dddr]\ar@{-}@/_25pt/[ddddddd]&2^3{\cdot}7^2\ar@{-}[dddddddr]\ar@{-}@/^20pt/[ddd]\ar@{-}[d]\ar@{-}[dddr]\ar@{-}@/_25pt/[ddddddd]
&2^3{\cdot}7^3\ar@{-}[dddddddr]\ar@{-}@/^20pt/[ddd]\ar@{-}[d]\ar@{-}[dddr]\ar@{-}@/_25pt/[ddddddd]&2^3{\cdot}7^4\ar@{-}[dddddddr]\ar@{-}@/^20pt/[ddd]\ar@{-}[d]\ar@{-}[dddr]\ar@{-}@/_25pt/[ddddddd]&\dots\ar@{-}\\
2^2{\cdot}7\ar@{-}[dddddr]\ar@{-}@/_20pt/[ddddd]\ar@{-}[d]&\ar@{-}[dddddr]2^2{\cdot}7^2\ar@{-}[dddddr]\ar@{-}@/_20pt/[ddddd]\ar@{-}[d]&2^2{\cdot}7^3\ar@{-}[dddddr]\ar@{-}@/_20pt/[ddddd]\ar@{-}[d]&2^2{\cdot}7^4\ar@{-}[dddddr]\ar@{-}@/_20pt/[ddddd]\ar@{-}[d]&\dots\ar@{-}\\
2{\cdot}7\ar@{-}[dddr]\ar@{-}@/_15pt/[ddd]\ar@{-}[d]&2{\cdot}7^2\ar@{-}[dddr]\ar@{-}@/_15pt/[ddd]\ar@{-}[d]&2{\cdot}7^3\ar@{-}[dddr]\ar@{-}@/_15pt/[ddd]\ar@{-}[d]&2{\cdot}7^4\ar@{-}[dddr]\ar@{-}@/_15pt/[ddd]\ar@{-}[d]&\dots\ar@{-}\\
7\ar@{-}[dr]\ar@{-}[d]&7^2\ar@{-}[dr]\ar@{-}[d]&7^3\ar@{-}[dr]\ar@{-}[d]&7^4\ar@{-}[dr]\ar@{-}[d]&\dots\ar@{-}\\
-7\ar@{-}[ur]\ar@{-}[d]\ar@{-}@/^20pt/[ddd]&-7^2\ar@{-}[ur]\ar@{-}[d]\ar@{-}@/^20pt/[ddd]\ar@{-}[dddl]&-\ar@{-}[ur]7^3\ar@{-}[d]\ar@{-}@/^20pt/[ddd]\ar@{-}[dddl]&-7^4\ar@{-}[ur]\ar@{-}[d]\ar@{-}@/^20pt/[ddd]\ar@{-}[dddl]&\dots\ar@{-}[dddl]\\
-2{\cdot}7\ar@{-}[uuur]\ar@{-}[d]\ar@{-}@/^20pt/[ddd]&-2{\cdot}7^2\ar@{-}[uuur]\ar@{-}[d]\ar@{-}@/^20pt/[ddd]\ar@{-}[dddl]&-2{\cdot}7^3\ar@{-}[uuur]\ar@{-}[d]\ar@{-}@/^20pt/[ddd]\ar@{-}[dddl]&-2{\cdot}7^4\ar@{-}[uuur]\ar@{-}[d]\ar@{-}@/^20pt/[ddd]\ar@{-}[dddl]&\dots\ar@{-}[dddl]\\
-2^2{\cdot}7\ar@{-}[uuuuur]\ar@{-}[d]&-2^2{\cdot}7^2\ar@{-}[uuuuur]\ar@{-}[d]&-2^2{\cdot}7^3\ar@{-}[uuuuur]\ar@{-}[d]&-2^2{\cdot}7^4\ar@{-}[uuuuur]\ar@{-}[d]&\dots\\
-2^3{\cdot}7\ar@{-}[uuuuuuur]\ar@{-}[d]&-2^3{\cdot}7^2\ar@{-}[uuuuuuur]\ar@{-}[d]&-2^3{\cdot}7^3\ar@{-}[uuuuuuur]\ar@{-}[d]&-2^3{\cdot}7^4\ar@{-}[uuuuuuur]\ar@{-}[d]&\dots\\
\vdots&\vdots&\vdots&\vdots&\\
}
$$
\caption{The graph $\Gamma_7$}\label{fig7}
\end{figure}
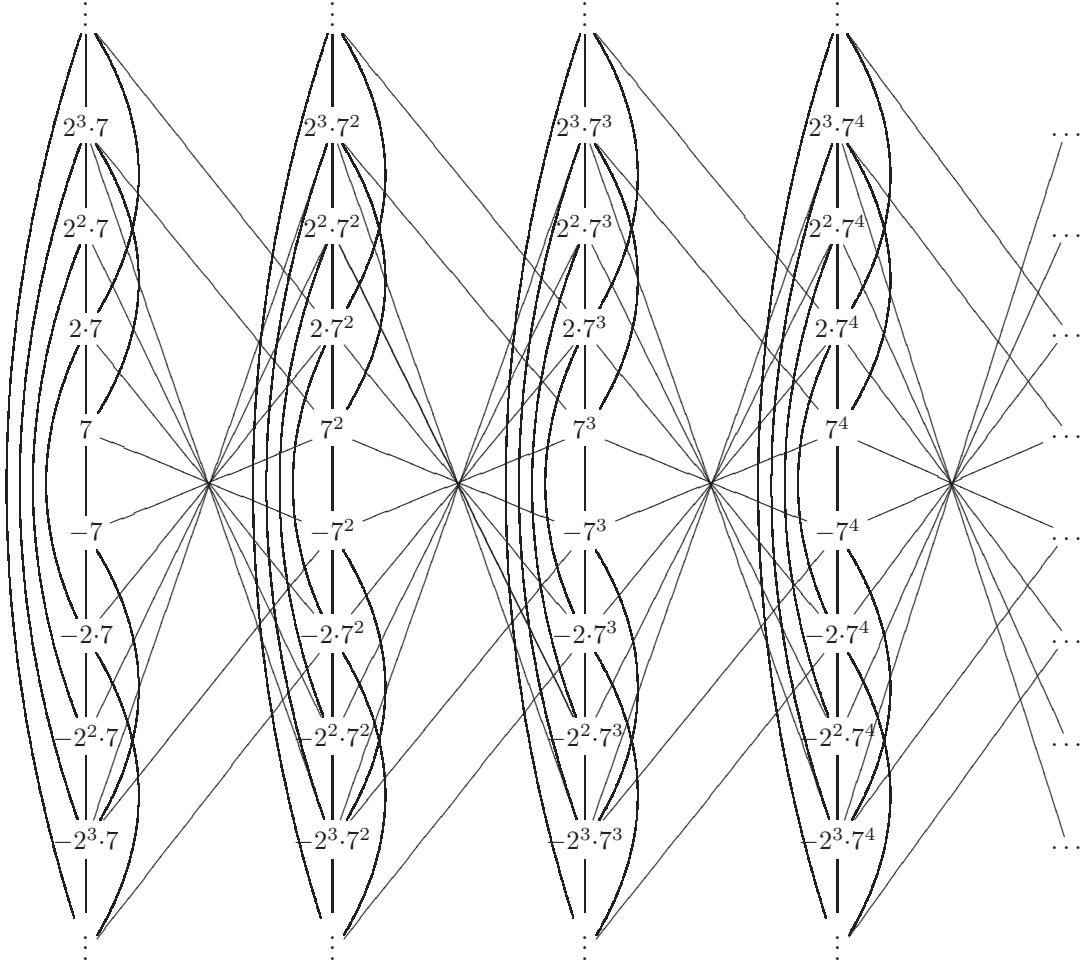

4. Assume that $p$ is not Fermat--Mersenne. It is easy to check that every doubleton
 $$\{x,y\}\in\big\{\{\varepsilon2^{a-1}p^b,-\varepsilon2^{a-1}p^b\}, \{\varepsilon2^{a-1}p^b,\varepsilon2^{a}p^b\}:a,b\in\IN, \ \varepsilon\in\{-1,1\}\big\}$$
 has $A_{\{x,y\}}=\Pi_x\cup\Pi_y\cup\Pi_{x-y}=\{2,p\}$ and hence $\{x,y\}\in\mathcal E_p$.

\begin{figure}
$$
\xymatrix@C40pt{
\ar@{-}[d]\ar@{-}@/^40pt/[ddddddddd]\vdots&\ar@{-}[d]\ar@{-}@/^40pt/[ddddddddd]\vdots&\ar@{-}[d]\ar@{-}@/^40pt/[ddddddddd]\vdots&\ar@{-}[d]\ar@{-}@/^40pt/[ddddddddd]\vdots\\
2^3{\cdot}11\ar@{-}[d]\ar@{-}@/^35pt/[ddddddd]&2^2{\cdot}11^3\ar@{-}[d]\ar@{-}@/^35pt/[ddddddd]&2^3{\cdot}11^3\ar@{-}[d]\ar@{-}@/^35pt/[ddddddd]&2^3{\cdot}11^4\ar@{-}[d]\ar@{-}@/^35pt/[ddddddd]&\cdots\\
2^2{\cdot}11\ar@{-}[d]\ar@{-}@/^30pt/[ddddd]&2^2{\cdot}11^2\ar@{-}[d]\ar@{-}@/^30pt/[ddddd]&2^2{\cdot}11^3\ar@{-}[d]\ar@{-}@/^30pt/[ddddd]&2^2{\cdot}11^4\ar@{-}[d]\ar@{-}@/^30pt/[ddddd]&\cdots\\
2{\cdot}11\ar@{-}[d]\ar@{-}@/^25pt/[ddd]&2{\cdot}11^2\ar@{-}[d]\ar@{-}@/^25pt/[ddd]&2{\cdot}11^3\ar@{-}[d]\ar@{-}@/^25pt/[ddd]&2{\cdot}11^4\ar@{-}[d]\ar@{-}@/^25pt/[ddd]&\cdots\\
11\ar@{-}[d]&11^2\ar@{-}[d]&11^3\ar@{-}[d]&11^4\ar@{-}[d]&\cdots\\
-11\ar@{-}[d]&-11^2\ar@{-}[d]&-11^3\ar@{-}[d]&-11^4\ar@{-}[d]&\cdots\\
-2{\cdot}11\ar@{-}[d]&-2{\cdot}11^2\ar@{-}[d]&-2{\cdot}11^3\ar@{-}[d]&-2{\cdot}11^4\ar@{-}[d]&\cdots\\
-2^2{\cdot}11\ar@{-}[d]&-2^2{\cdot}11^2\ar@{-}[d]&-2^2{\cdot}11^3\ar@{-}[d]&-2^2{\cdot}11^4\ar@{-}[d]&\cdots\\
-2^3{\cdot}11\ar@{-}[d]&-2^3{\cdot}11^2\ar@{-}[d]&-2^3{\cdot}11^3\ar@{-}[d]&-2^3{\cdot}11^4\ar@{-}[d]&\cdots\\
\vdots&\vdots&\vdots&\vdots&\\
}
$$
\caption{The graph $\Gamma_{11}$}\label{fig11}
\end{figure}
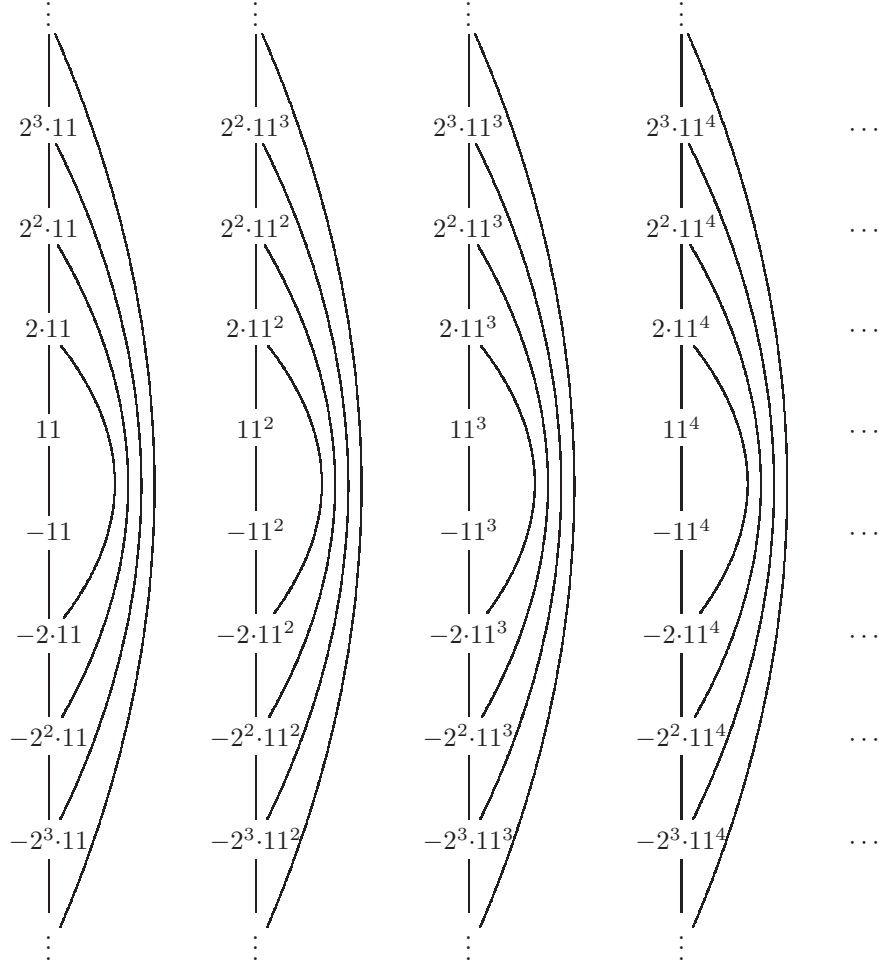

Now assume that $\{x,y\}\in\mathcal E_p$ is an edge of the graph $\Gamma_p$. Then $\Pi_x\cup\Pi_y\cup\Pi_{x-y}=A_{\{x,y\}}=\{2,p\}$ and $\{x,y\}$ can be written as $\{\varepsilon2^{a-1}p^b,\delta2^{c-1}p^d\}$ for some $a,b,c,d\in\IN$, $\varepsilon,\delta \in\{-1,1\}$ with $2^{a-1}p^b\le 2^{c-1}p^d$.

If $a=c$ and $b=d$, then $\varepsilon\ne\delta$ and  $\{x,y\}=\{2^{a-1}p^b,-2^{a-1}p^b\}$.

If $a=c$, then $b\le d$ and the inclusion $\Pi_{x-y}\subseteq\{2,p\}$ implies that $\Pi_{p^{d-b}-\varepsilon/\delta}\subseteq\{2,p\}$ and hence $p^{d-b}-\varepsilon/\delta$ is a power of $2$. By the Mih\u ailescu Theorem~\ref{Mihailescu}, $d-b=1$ and hence $p$ is a Fermat prime or Mersenne prime which is not true.

If $b=d$, then $a\le c$ and the inclusion  $\Pi_{x-y}\subseteq\{2,p\}$ implies that $\Pi_{2^{c-a}-\varepsilon/\delta}\subseteq\{2,p\}$ and hence $2^{c-a}-\varepsilon/\delta$ is a power of $p$. By the Mih\u ailescu Theorem~\ref{Mihailescu}, $2^{c-a}-\varepsilon/\delta\in\{1,p\}$. Taking into account that $p$ is not Fermat or Mersenne prime,  we conclude that if $\varepsilon=\delta$, $2^{c-a}-1=1$ and hence $c-a=1$. Then $\{x,y\}=\{\varepsilon2^{a-1}p^b,\varepsilon2^{a}p^{b}\}$.

So, we assume that $a\ne c$ and $b\ne d$. By analogy with the case of Fermate primes, we can show that the subcases ($a<c$ and $b<d$) and ($a>c$ and $b>d$) are impossible. 	

If $a<c$ and $b>d$, then $\Pi_{x-y}\subseteq\{2,p\}$ implies that $2^{c-a}-(\varepsilon/\delta) p^{b-d}=1$. By the Mih\u ailescu Theorem~\ref{Mihailescu}  $b-d=1$ and hence $p=2^{c-a}-1$ is a Mersenne prime, which is not true.

If $a>c$ and $b<d$, then $\Pi_{x-y}\subseteq\{2,p\}$ implies that $p^{d-b}-(\varepsilon/\delta)2^{a-c}=1$. By the Mih\u ailescu Theorem~\ref{Mihailescu}  $d-b=1$ and hence $p=1+2^{a-c}$ is a Fermat prime, which is not true.	
\end{proof}

In Figures~\ref{fig3}, \ref{fig5}, \ref{fig7}, \ref{fig11} we draw the graphs $\Gamma_p$ for $p$ equal to $3,5,7,11$. Observe that $3$ is both Fermat and Mersenne prime, $5$ is Fermat prime, $7$ is Mersenne prime and $11$ is not Fermat--Mersenne.

\begin{lemma}\label{l:graph}
%For every homeomorphism $h$ of the Kirch space and any prime number $p$ which is not Fermat--Mersenne we have $h[\pm p^\IN]=\pm p^\IN$ where $\pm p^\IN=\{\varepsilon p^n:n\in\IN, \varepsilon\in\{-1,1\}\}$. Moreover, $h(\pm p)=\pm p$ whenever $p$ is Fermat--Mersenne prime $p\geq3$.
 Let $p$ be an odd prime number and $h$ be a positive homeomorphism of the Kirch space.
\begin{enumerate}
\item If $p$ is Fermat-Mersenne, then $h(p)=p$;
\item If $p$ is not Fermat-Mersenne, then $h(p)=p^n$ for some $n\in \IN$.
\end{enumerate}
\end{lemma}

\begin{proof}
 1. Lemma~\ref{struct}(1) implies that the degree of $\pm3$ in the graph $\Gamma_3$ is equal to $8$ but the other vertices have  degree at least $9$. Hence $h(3)=\pm3$. Assume that $h(3)=-3$.	Then  by Lemma \ref{l:2fix} and by Lemma \ref{lemmma} $\{2,3\}=A_{\{2,3\}}=A_{h(\{2,3\})}=A_{\{2,-3\}}=\{2,3,5\}$ but this is not true and hence $h(3)=3$.

Assume that  $p>3$ is Fermat or Mersenne prime. Lemma~\ref{struct}(2,3) implies that the degree of $\pm p$ in the graph $\Gamma_p$ is $4$ but the other vertices have  degree at least $5$. Hence $h(p)=\pm p$. Assume that $h(p)=- p$. By Lemma \ref{l:2fix}, $A_{\{1,p\}}=A_{\{1,h(p)\}}=A_{\{1,-p\}}$, so $\{p\}\cup \Pi_{p-1}=\{p\}\cup\Pi_{p+1}$, according to Lemma \ref{2ae}. This implies that $\Pi_{p-1}=\Pi_{p+1}=\{2\}$. Hence $p$ is both Fermate and Mersenne which  is possible iff $p=3$ and this contradicts our assumption. Therefore $h(p)=p$.

 2. Let $p$ be an odd prime number, which is not Fermat--Mersenne.
Lemma~\ref{struct}(4) implies that the set $\pm p^\IN=\{\varepsilon p^n:n\in\IN, \varepsilon\in\{-1,1\}\}$ coincides with the set of vertices of order 2 in the graph $\Gamma_p$. Taking into account that $h{\restriction}V_p$ is an isomorphism of the graph $\Gamma_p$, we conclude that $h( p)=\pm p^n$ for some $ n\in \IN$.  Assume that $h(p)=-p^n$. Then $h(\{-1,p\})=\{-1,-p^n\}$. By Lemma \ref{lemmma}, $A_{\{-1,p\}}=A_{\{-1,-p^n\}}$, so $\{p\}\cup \Pi_{p+1}=\{p\}\cup\Pi_{p^n-1}$, according to Lemma \ref{2ae}. Since $\{p\}$ does not intersect $\Pi_{p+1}$ and $\Pi_{p^n-1}$ we conclude that $ \Pi_{p+1}=\Pi_{p^n-1}$. Hence we get the inclusion $\Pi_{p-1}\subseteq\Pi_{p^n-1}=\Pi_{p+1}$. If some prime number $d$ divides $p-1$ then the inclusion $\Pi_{p-1}\subseteq\Pi_{p+1}$ implies that $d$ divides ${p+1}$, consequently  $d$ divides the difference $(p+1)-(p-1)=2$ and hence $d=2$. As a consequence, $\Pi_{p-1}=\{2\}$ and $p-1=2^m$ for some $m\in \IN$ which contradicts the assumption that $p$ is not Fermat prime. Hence $h( p)= p^n$.
\end{proof}

%to prove that $h(p)=p$ for any prime number $p$, we  will need the following lemma.

\begin{lemma}
\label{l:F1x} For any positive homeomorphism $h$ of the Kirch space and any prime number $p$ we have $h(p)=p$.
% For any integer number $x\in\IN\setminus\{1\}$ the filter $\F_{\{1,x\}}$ is the greatest element of the subset $$\mathfrak F_x=\{\F_{\{1,x^n\}}:n\in\IN\}$$ of the poset $\mathfrak F$. If $x\notin \{2m:m\in\IN\}\cup\{2^m-1:m\in\IN\}$, then $\{n\in\IN:\F_{\{1,x^n\}}=\F_{\{1,x\}}\}=\{1\}$.
\end{lemma}

\begin{proof}
If $p=2$, then $h(p)=p$ by Lemma~\ref{l:2fix}. If $p$ is Fermat--Mersenne, then $h(p)=p$ by Lemma~\ref{l:graph}. So, we assume $p$ is not Fermat--Mersenne. By Lemma~\ref{l:graph}, $h(p)=p^n$ for some $n\in\IN$.
By Lemmas~\ref{2ae}, \ref{l:2fix} and \ref{lemmma}, $$\{p\}\cup \Pi_{p-1}=A_{\{1,p\}}=A_{\{1,h(p)\}}=A_{\{1,p^n\}}=\{p\}\cup\Pi_{p^n-1}$$and hence $\Pi_{p^n-1}=\Pi_{p-1}$. Since $p$ is not Mersenne prime, Zsigmondy Theorem~\ref{Zsigmondy} guarantees that $n=1$ and hence $h(p)=p^1=p$.% Observe that for every $n\in\IN$ the number $x-1$ divides $x^n-1$, which implies $$A_{\{1,x\}}=\{2\}\cup\Pi_x\cup\Pi_{x-1}\subseteq\{2\}\cup\Pi_{x^n}\cup\Pi_{x^n-1}=A_{\{1,x^n\}}.$$Observe also that $\Pi_{\{1,x\}}=\emptyset=\Pi_{\{1,x^n\}}$ and $\alpha_{\{1,x\}}(p)=1=\alpha_{\{1,x^n\}}(p)$ for every $p\in A_{\{1,x\}}$. By Lemma~\ref{l:wo2}, $\F_{\{1,x^n\}}\subseteq\F_{\{1,x\}}$, which means that $\F_{\{1,x\}}$ is the largest element of the poset $\mathfrak F_x$.
%Now assume that $x\notin\{2m:m\in\IN\}\cup\{2^m-1:n\in\IN\}$ and $\F_{\{1,x\}}=\F_{\{1,x^n\}}$ for some number $n$. We should prove that $n=1$. To derive a contradiction, assume that $n\ge 2$. By Lemmas~\ref{l:wo2} and \ref{2ae}, $$\{2\}\cup\Pi_x\cup\Pi_{x^n-1}=A_{\{1,x^n\}}=A_{\{1,x\}}=\{2\}\cup\Pi_x\cup\Pi_{x-1}$$ and hence $\Pi_{x^n-1}\subseteq\{2\}\cup\Pi_{x-1}=\Pi_{x-1}\subseteq\bigcup_{0<k<n}\Pi_{x^k-1}$. By Zsigmondy Theorem~\ref{Zsigmondy}, $x\in \{2\}\cup\{2^m-1\}_{m\in\IN}$, which contradicts our assumption.
\end{proof}%\begin{lemma}\label{l:pfix} For any homeomorphism $h$ of the Kirch space and any prime number $p$ we have $h(p)=p$.
%\end{lemma}
%\begin{proof} If $p=2$, then $h(p)=p$ by Lemma~\ref{l:2fix}. If $p$ is Fermat--Mersenne, then $h(p)=p$ by Lemma~\ref{l:graph}. So, we assume $p$ is not Fermat-Mersenne. By Lemma~\ref{l:graph}, $h[p^\IN]=p^\IN$, which implies $\tilde h[\mathfrak F_p]=\mathfrak F_p$ where
%$$\mathfrak F_p=\{\F_{\{1,p^n\}}:n\in\IN\}.$$ By Proposition~\ref{p:iso}, $\tilde h$ induces an order isomorphism of the poset $\mathfrak F_p$ (endowed with the inclusion order, inherited from the poset $\mathfrak F$).
 %By Lemma~\ref{l:F1x}, $n=1$ is a unique number such that $\F_{\{1,p^n\}}$ coincides with the greatest element $\F_{\{1,p\}}$ of the poset $\mathfrak F_x$. This order characterization of the filter $\F_{\{1,p\}}$ implies that $h(p)=p$.
 %\end{proof}
\begin{lemma} The positive homeomorphism group of the Kirch space is trivial.
\end{lemma}

\begin{proof} To derive a contradiction, assume that the Kirch space admits a homeomorphism $h$ such that $h(x)\ne x$ for some number $x$. By the Hausdorff property of the Kirch space and the continuity of $h$, there exists a neighborhood $O_x$ of $x$ in the Kirch topology such that $h[O_x]\cap O_x=\emptyset$. By the definition of the Kirch topology, there exists a square-free number $b$ such that $\Pi_b\cap \Pi_x=\emptyset$ and $x+b\IZ\subseteq O_x$. By the  Dirichlet Theorem~\ref{Dirichlet}, the arithmetic progression $x+b\IN\subseteq x+b\IZ$ contains some prime number $p$. Then $h[O_x]\cap O_x=\emptyset$ implies $h(p)\ne p$, which contradicts Lemma \ref{l:F1x}.
\end{proof}
Our final lemma completes the proof of Theorem~\ref{t:main}.

\begin{lemma}Any homeomorphism $h$ of the Kirch space $ \IZ^\bullet$ is equal to $i:\IZ^\bullet\to\IZ^\bullet$, $i:x\mapsto x$ or to $j:\IZ^\bullet\to\IZ^\bullet$$, j:x\mapsto-x$.
\end{lemma}

\begin{proof} If $h$ is positive, then $h=i$ by previous Lemma. If $h$ is not positive then $h(1)<0$ and $j\circ h(1)>0$. Then the homeomorphism $j\circ h$ is positive and equals $i$ by the preceding case. This implies that $$h=i\circ h =(j\circ j)\circ h=j\circ( j \circ h)=j\circ i=j.$$
\end{proof}

\noindent{\bf Acknowledgement.} The author expresses her sincerely thanks to  Taras Banakh for his generous help during preparation of this paper.
%\newpage

\end{document}